\documentclass[11pt]{article}

\usepackage{a4wide}
\usepackage{tikz,overpic,graphicx,caption,subcaption,pgfplots}
\usetikzlibrary{arrows.meta}
\pgfplotsset{compat=1.13}

\usepackage{amsmath,amsthm,amssymb}
\usepackage{mathtools}
\usepackage{bbm}

\title{A double scaling limit for the d-PII equation with boundary conditions}

\author{Maurice Duits\footnote{Department of Mathematics, Royal Institute of Technology, Lindstedtsvägen 25, SE 10044, Stockholm Sweden. E-mail: duits@kth.se} 
\and Diane Holcomb\footnote{Department of Mathematics, Royal Institute of Technology, Lindstedtsvägen 25, SE 10044, Stockholm Sweden. E-mail: holcomb@kth.se}}
\date{}

\theoremstyle{plain}
\newtheorem{theorem}{Theorem}[section] 
\newtheorem{lemma}[theorem]{Lemma}
\newtheorem{proposition}[theorem]{Proposition}
\newtheorem{corollary}[theorem]{Corollary}

\newtheorem{remark}[theorem]{Remark}
\usepackage{amsmath,amsthm,amssymb}
\usepackage{color}
\usepackage[shortlabels]{enumitem}
\usepackage{geometry}
\usepackage{hyperref}

\renewcommand{\Re}{\mathop{\mathrm{Re}}}

\newcommand{\azero}{a^{(0)}}
\newcommand{\y}{u} 
\newcommand{\yy}{y} 
\newcommand{\Y}{\nu}
\newcommand{\Ai}{\mathop{\mathrm{Ai}}}
\newcommand{\Rdelta}{\delta}
\newcommand{\diag}{\mathrm{diag}}

\newcommand{\rootroot}{\sqrt{\frac{\sqrt{1-4x^3}-1}{\sqrt{1-4x^3}+1}}}
\DeclareMathOperator*{\Tr}{\mathrm{Tr}}
\numberwithin{equation}{section}
\newcommand{\shift}{\varepsilon}

\begin{document}

\maketitle

\begin{abstract}
    We study a double scaling limit for a solution of the discrete Painlev\'e II equation with boundary conditions. The location of the right boundary point is in the critical regime where the discrete Painlev\'e equation turns into the continuous Painlev\'e II equation. Our main results it that, instead of the Hastings-McLeod solution (which would occur when the right boundary point is at infinity), the solution to the discrete equation converges in a double scaling limit   to a tronqu\'ee solution of the Painlev\'e II equation that behaves like the  Hastings-McLeod solution at minus infinity and has a pole at a prescribed location. Our proof of the double scaling limit is based on finding an  approximation that is sufficiently close in order to apply the  Kantorovich theorem for Netwons method. To meet the criteria for this theorem, we will establish a lower bound for the  solutions to the Painlev\'e II equation that occur (including the Hastings-McLeod solution).
 \end{abstract}
\setcounter{tocdepth}{1}
\tableofcontents

\section{Introduction}

In this paper we study the solution $(a_k)_{k=0}^{n-2}$ to the discrete Painlev\'e II equation (d-PII) 
\begin{equation}\label{eq:dPII}
    a_{k+1}+a_{k-1}
        =\frac{k+\shift}{t} \frac{2a_k}{1-a_k^2},\qquad k=0,\ldots,n-2,
\end{equation}
with $t,\shift>0$ and $n \in \mathbb N$, that is characterized by the boundary conditions 
\begin{equation} \label{eq:BC}
a_{-1}=a_{n-1}=1.
\end{equation}
Note that it is not a priori obvious that such a solution exists, but we will show that it does and is unique. Moreover, we will show that  it  converges,  as $n \to \infty$ in  a particular double scaling limit, to a  special solution of the continuous Painlev\'e II equation (PII)
\begin{equation}
    \label{eq:PII2}
    \yy''(x) = x\yy(x) + 2 \yy^3(x).
\end{equation}
 As we will see, the boundary condition $a_{n-1}=1$ will have a strong impact  and our  double scaling limit differs from established results in the literature in that we observe different solutions to the PII equation.   The solutions that we encouter are characterized by their behavior for $x \to -\infty$ and the location of a left-most pole on $\mathbb R$. We will prove the existence of such solutions and  derive further  properties along the way. But before we come to the precise statements of our results in Section \ref{sec:state} we first recall some historical background on  double scaling limits for the d-PII equation and motivate the boundary conditions \eqref{eq:BC}. 

\subsubsection*{Painlev\'e equations}

The PII equation \eqref{eq:PII2} is the second of six Painlev\'e equations. The story behind these equations starts around  1900 with the works of Painlev\'e, and later Gambier,   who sought to classify all non-linear second order differential equations that satisfy the Painlev\'e property: the only moveable singularities are poles or branch points. Ultimately, the list of such equations can be reduced to six non-linear equations that we nowadays refer to as the Painlev\'e equations. Their major breakthrough came much later after the discovery that they appear  as ODE reductions of integrable non-linear PDE's and play an important role random matrix theory. The PII equation \eqref{eq:PII2} appears as the ODE reduction of the Korteweg-de Vries equation \cite{AbSe}. The celebrated Tracy-Widom distibution \cite{TW} can be expressed in terms of the Hastings-McLeod solution \cite{HM} for the PII equation. That same solution also appears in  the local eigenvalue correlations for Hermitian random matrices  near critical point where the limiting density vanishes quadratically \cite{BI,CK}.  In fact, our main motivation for studying the d-PII with boundary condition \ref{eq:BC} is to ultimately being able to describe such local correlations for  $\beta$-ensembles (the Hermitian models correspond to $\beta=2$), which we will explain in more detail below. For general backgound on Painlev\'e transcendents  including further appearance in mathematical phyciscs we refer to \cite{Clarkson,Conte,FIKN}.

The discrete Painlev\'e equations also have a rich structure with important application to mathematical physic and we refer to \cite{Joshi,VanAssche1} as general references. Their characterization is more complicated than the continuous one, but they can been classified using rational surfaces and affine root systems \cite{Sakai}.  Just as its continuous analogue, the d-PII equation \eqref{eq:dPII} is a reduction of the discrete modified Korteweg-de Vries equation \cite{Nij,Nijhoff}. The discrete Painlev\'e equations  also make an important appearance though the study of matrix integrals in string theory \cite{BK,DS,GM}. For instance, the d-PII  was first studied in this context in  \cite{PS} for an integral over unitary matrices.  

An important property of  discrete Painlev\'e equations is  that,  in certain double scaling limits, they converges to continuous ones (at least formally). For instance, after making inserting the ansatz  \begin{equation} \label{eq:ansatz}
a_{n+k}=n^{-\frac 1 3}  \y\left({k} {n^{-\frac 13}}\right)
\end{equation}
into  \eqref{eq:dPII}  for a  smooth function $\y$ and taking the limit $n\to \infty$ such that $$t=n, \quad k/n^{\frac 13} \to  x,$$ it is not hard to see that the d-PII in  \eqref{eq:dPII} turns into the PII  \eqref{eq:PII2} with  $\y(x)=2^{\frac 13}\yy(2^{\frac 13}x)$. This, of course, is a formal procedure and there is no a priori mechanism that guarantees that solutions to \eqref{eq:dPII} converge to solution of \eqref{eq:PII2}. In fact, it is a major task to make such limits rigorous. One popular direction comes from the theory of orthogonal polynomials and although we do not follow this approach in this paper, we will give a brief overview of some of its most important developments.

\subsubsection*{Orthogonal polynomials}

 Orthogonal polynomials satisfy recurrence relations and for special integrable choices of the orthogonality measure, the recurrence coefficients in these relations themselves satisfy difference equations. Various explicit examples are known in the literature in which these difference equations are particular cases of discrete Painlev\'e equations.  We refer to \cite{VanAssche1} for a good survey of the topic. 

The d-PII equation \eqref{eq:dPII}  can be obtained by considering polynomial on the unit circle \cite{PS}. For $k=0,1,\ldots,$ $\Phi_k(z)$ be the unique monic polynomial of degree $k$  satisfying
\begin{equation} \label{eq:OPUCpii}
    \int_{0}^{ 2 \pi} \Phi_k(e^{i \theta}) e^{-\ell \theta} e^{-t \cos \theta} d\theta=0, \qquad \ell=0,\ldots, k-1. 
\end{equation}
Then there exists $\{\alpha_k\}_{k=0}^\infty$, often referred to as Verblunsky coefficients, such that if we simultaneously define the conjugate polynomials $\Phi_k^*(z) = z^k \overline \Phi_k(1/z)$. The families $\Phi$ and $\Phi^*$ obey the recurrence
\begin{align} \label{eq:Verblunsky}
\Phi_{k+1} &= z \Phi_k(z) - \bar \alpha_k \Phi_k^*(z) \\
\Phi^*_{k+1}& = \Phi_k^*(z) - \alpha_k z \Phi_k(z) \notag
\end{align}
with initial condition $\Phi_0(z)= \Phi_0^*(z)=1$. It  is not difficult to verify that \cite{PS} the Verblunsky coefficients satisfy \eqref{eq:dPII}  with $\shift=1$   for $k=0,1, \ldots$. In fact, they are the unique solution with the properties 
\begin{equation} \label{eq:HMconditions}
    \alpha_{-1}=1, \qquad |\alpha_k| < 1.
\end{equation}
See \cite{VanAssche1} for a discussion and a proof.  Note that instead of the boundary conditions \eqref{eq:BC} that we are interested in, the Verblunsky coefficients  for the orthogonality weight \eqref{eq:OPUCpii} have the condition that they are bounded (which can be interpreted as condition at infinity).

In the past 25 years,  a vast amount of literature has been devoted to the asymptotic study of polynomials that satisfy orthogonality relations. An important driving force behind this are the steepest descent techniques developed in \cite{DKMVZ1,DKMVZ2}  for the Rieman-Hilbert problem that characterizes the orthogonal polynomials \cite{FIK}.  In a celebrated work \cite{BDJ} on the longest increasing subsequence, Baik, Deift and Johansson used this approach to study the orthogonal polynomials on the unit circle  with varying weight $e^{-n \tau \cos \theta} d\theta$ and showed the connection to the PII equation. Although it was not the focus of \cite{BDJ}, the convergence of the Verblunsky coefficients to the Hastings-McLeod solution of the PII equation, in a double scaling limit, can  be derived from their analysis.

Similar Painl\'eve II asymptotics have also been  computed for orthogonal polynomials on the real line.  Indeed, Bleher and Its \cite{BIann,BI} computed the asymptotic behavior of the recurrence coefficients for orthogonal polynomials on the real line with a double well potential. These coefficients satisfy the so-called discrete string equation or the discrete Painlev\'e I equation.  In \cite{BI}  it was shown that they converge,  a double scaling limit, to the Hastings-McLeod solution of the PII equation (they also describe the relation between their work and \cite{BDJ}). After that, Claeys and Kuijlaars \cite{CK} proved that this is a universal behavior and occurs  whenever the equilibrium measure is regular except at an interior point where the density vanishes quadratically.  In a subsequent work \cite{CKV} Claeys, Kuijlaars and Vanlessen proved that one even can obtain the general form of the PII equation by adding a logarithmic singularity to the potential located at the irregular point of the equilibrium measures.  It is also interesting to note that  the discrete string equation was originally motivated by a different double scaling limit  to the Painlev\'e I equation \cite{BK,DS,GM} (instead of PII), but for a choice of parameters where the Hermitian matrix integral is divergent. By considering polynomials that satisfy orthogonality on contours in the plane one can still obtain solutions in terms of recurrence coeffiicents and rigorously study the double scaling limit \cite{DK,FIK}.

As mentioned above, an important difference with these works is that, in this paper, we consider a boundary value problem. Although there in principle exists an orthogonality measure on the circle such that the Verblunsky coefficients solve the boundary value problem, but it is not easy to find this measure explicitly (the boundary condition forces the support to consist of finitely many points). We will therefore not use orthogonal polynomials or Riemann-Hilbert techniques in our proofs. 

 We also stress that the right boundary point in \eqref{eq:BC} is in the regime where we expect convergence to a solution to the continuous PII equation. As we will see, this means that our solution to the d-PII equation will not  converge to the Hastings-Mcleod solution  to the PII equation, but to a different solution that has real poles.  The solutions that we are interested in  have the same asymptotic  behavior at $–\infty$ as the Hastings-McLeod solution and are thus all examples of so-called tronqu\'ee solutions.

\subsubsection*{Closing of the gap for $\beta$-ensembles}

Our motivation for studying the d-PII equation \eqref{eq:dPII} with boundary conditions \eqref{eq:BC} comes from $\beta$-ensembles with a closing gap. For $\beta,\tau>0$ consider the jpdf on $[0,2 \pi ]^n$  proportional to  
\begin{equation} \label{eq:CbetaE}
 \prod_{1\leq j < k \leq n} |e^{i \theta_j}-e^{i \theta_k}|^{\beta} \prod_{\ell=1}^ne^{- \frac12 n \tau \beta \cos \theta_\ell } d\theta_\ell.
\end{equation}
Note that for $\tau=0$ this is the Circular $\beta$-ensemble (C$\beta$E), which is a classical model from random matrix theory. For that choice, the probability measure is rotationally invariant. When $n \to \infty$, the points $e^{i \theta_j}$ for $j=1,\ldots,n$ will be a (random) perturbation of $n$ equidistant points on the unit circle. After zooming in at scales $\sim 1/n$ near an arbitrary point on the unit circle, the random points locally will converge to the Sine$_\beta$ -process. For $\beta=2$ this is the determinantal point processes with kernel $k(\theta, \eta) = \sin\big( \pi(\theta-\eta)\big)/(\theta-\eta)$. For general $\beta$ a characterization is more complicated, and the first one was given by Killip and Stoiciu in \cite{KS}. A later characterization of the process as the eigenvalues of a random Dirac operator was shown by Valk\'o and Vir\'ag in \cite{BVBV}. 

For $\tau>0$ the effect of the term $e^{-n \tau \beta \cos  \theta_\ell}$ is that the points are less likely to be close to $\theta=0$.  They will have a limiting density on the circle, but it is no longer uniform. To understand the limit, it helps to rewrite \eqref{eq:CbetaE} as the Gibbs measure
$$
e^{-\frac{n^2\beta}{2} \mathbb H(\theta_1,\ldots,\theta_n)} d\theta_1\ldots d \theta_n.
$$
where 
$$
	\mathbb H(\theta_1,\ldots,\theta_n)=\frac{1}{n^2}\sum_{j \neq k } \log \left| e^{i \theta_j}-e^{i \theta k}\right|^{-1}+ \frac{1}{n}\sum_{j=1}^n   \tau  \cos \theta_j.
$$
This allows us to think of the $e^{i \theta_j}$ as a Coulomb gas in the presence of an external field $2 \tau \cos \theta$. As $n \to \infty$, the random configuration will converge to the equilibrium measure which can be computed explicitly.  For $0<\tau\leq 1$, the points will still fill out the full circle but with  limiting density \cite[Lemma 4.3]{BDJ}
$$
	(1-\tau \cos \theta) \frac{d\theta}{2 \pi}.
$$
Note that this density is indeed no longer uniform, but dented near $\theta=0$. When $\tau>1$ this dent has turned into a gap and there will be no points near $\theta=0$ (with overwhelming probability).  The transition happens at $\tau=1$. In that case the limiting density will vanish quadratically at $\theta=0$. Near such a point the local process will no longer be described by the $\sin_\beta$ processes. Indeed, for $\beta=2$ the local scaling limit will converge to the determinantal point process where the correlation kernel is expressed in terms of $\Psi$-functions coming from the the Lax-pair for the Hastings-McLeod solution of the PII equation \cite{BI,CK}. For $\beta\neq 2$ the local process is unknown and this is an important open problem that was the motivation for this paper. 

\subsubsection*{Random matrices and d-PII with boundary conditions}

The key to characterizing the local process near the closing of a gap for $\beta$-ensembles is to study the Szeg\"o recursion associated to a random measure $\mu_n = \sum_{k=1}^n q_k \chi_{x = e^{i\theta_k}} $ on the unit circle where $(q_1, q_2,...q_n) \sim $ Dirichlet$(\frac{\beta}{2},..., \frac{\beta}{2})$.  

The Verblunsky coefficients for this measure arise naturally when you characterize the point process as the spectra of a random operator as Killip and Nenciu do in \cite{KN}. The first step is to identify a random matrix model for which the eigenvalues are distributed according to \eqref{eq:CbetaE}. To this end, for $\alpha_j \in \mathbb C$ for $j=0,\ldots, n-1$ such that $|\alpha_0|,\ldots, |\alpha_{n-2}| <1$ and $|\alpha_{n-1}|=1$, define $\rho_k=\sqrt{1-|\alpha_k|^2}$ and 
$$
\Xi_k=\begin{pmatrix} \overline{\alpha_k} & \rho_k\\ \rho_k & -\alpha_k
\end{pmatrix}.
$$ 
Set $\Xi_{-1}=[1]$ and $X_{n-1}=[\overline{\alpha_{n-1}}]$ as $1 \times 1$ matrices. Then let  $U$ be the $n \times n$  matrix 
$$U= \diag(\Xi_0,\Xi_2,\ldots) \diag (\Xi_{-1},\Xi_1,\ldots).$$ 
Then $U$ is a unitary matrix and 
if we choose $\alpha_0,\ldots, \alpha_{n-2}$ and  $\alpha_{n-1}=e^{i \phi}$ randomly   from the joint probability density 
\begin{equation}
    \label{eq:densityVerblunsky}
   e^{t \beta \left( -\Re \alpha_0+ \sum_{k=0}^{n-2}\Re \overline{ \alpha_k }\alpha_{k-1}\right)  }\prod_{\ell=0}^{n-2} (1-|\alpha_\ell|^2)^{\frac{\beta}{2}(n-1-\ell)-1}  d \alpha_\ell \ d\phi
\end{equation}
then the eigenvalues $e^{i \theta_j}$ of $U$ will be distributed according to \eqref{eq:CbetaE} with $t=\tau n$. Indeed, this follows from computing the Jacobian for the map from the Verblunksky coefficients to the spectral measure \cite[Lemma 4.1 and Proposition 4.2]{KN} and the simple computation
$$
    \sum_{\ell=1}^n \cos \theta_\ell= \frac12 \Tr \left(U+ U^*\right)=\Re \alpha_0- \sum_{\ell=0}^{n-2}\Re \overline{ \alpha_\ell }\alpha_{\ell+1}.
$$
Observe that in the special case $t = 0$ there is no correlation between the associated Verblunsky coefficients. The type of dependence when $ t \ne 0$ also appears in the recurrence coefficients for $\beta$-ensembles on $\mathbb{R}$ in the case where the potential is quartic (or higher order). Local limits in this type of model were studied by Krishnapur, Rider, and Vir\'ag as part of showing universality of the edge process \cite{KRV}.

 When $n \to \infty$ it is natural to ask for the ``typical" behavior of the Verblunsky coefficients $\alpha_k$. To this end, write  \eqref{eq:densityVerblunsky} as 
$$
    e^{-\frac{\beta  t}{2} \mathcal H(\alpha)}\  \prod_{\ell=0}^{n-2} d\phi d\alpha_\ell,
$$
with 
\begin{equation}
\label{eq:alphaHamiltonian}
    \mathcal H=- \sum_{\ell=0}^{n-2} \frac{n-1-\ell-\frac{2}{\beta}}{t}\log(1-|\alpha_\ell|^2)+\Re \alpha_0- \sum_{\ell=0}^{n-2}\Re \overline{ \alpha_\ell }\alpha_{\ell+1}.
\end{equation}
From this form we can see that we should expect the Verblunsky coefficients to concentrate near the minimizer of $\mathcal H$ (recall that we scale $t=n \tau$). Observe that the arguments for the minimizer are easy to find. Indeed,
$$
\mathcal H\geq - \sum_{\ell=0}^{n-2} \frac{n-1-\ell-\frac{2}{\beta}}{t}\log(1-|\alpha_\ell|^2)-|\alpha_0|- \sum_{\ell=0}^{n-2}|\alpha_\ell| |\alpha_{\ell+1}|.
$$
with equality if and only if $\alpha_j=|\alpha_j|$ for $j=0,\ldots,n-1$ (and thus, in particular, $\alpha_{n-1}=1$). By setting $|\alpha_{\ell}|=a_{n-2-\ell}$ and $\shift = 1- \frac{2}{\beta}$ we can write the right-hand side of \eqref{eq:alphaHamiltonian} as 
\begin{equation}
\label{eq:aHamiltonian}
H_n(a)= - \sum_{k=0}^{n-2} \frac{k+ \shift}{t}\log(1-a_k^2)- \sum_{k=-1}^{n-2} a_{k}a_{k+1}
\end{equation}
where we have also set 
$$
a_{-1}=a_{n-1}=1.
$$
Any solution of $\nabla H_n=0$ will be a local minimizer.  A straightforward computation shows that $\nabla H=0$ is precisely the d-PII equation \eqref{eq:dPII} with boundary conditions \eqref{eq:BC}.  Studying the solution and its scaling limit is thus an important first step in constructing a random operator that whose spectra is the local process near the closing of a gap for $\beta$-ensembles.

\subsubsection*{Acknowledgements}

M. Duits was partially supported by the Swedish Research Council (VR), grant no. 2016-05450 and grant no. 2021-06015, and the European Research Council (ERC), Grant Agreement No. 101002013. D. Holcomb was partially supported by the Knut and Alice Wallenberg foundation (KAW), grant no. KAW 2015.0359, and Swedish Research Council (VR), grant no. 2018-04758. 

We are grateful to Sara Zahedi for discussions on the Kantorovich Theorem for Newtons method and pointing us to \cite{Yam}.

\section{Statement of results} \label{sec:state}
In this Section we state our main results. The proofs will be postponed to later section.

 Our first result is that the solution to the boundary value problem for the d-PII  exists and is unique. 
\begin{theorem} \label{thm:last}
	Let $t,\shift>0$ and $n \in \mathbb N$.
	There exists a unique solution $(a_k)_{k=0}^{n-2} \in (0,1)^{n-1}$ to the d-PII equation \begin{equation}\label{eq:dPIIaa}
	a_{k+1}+a_{k-1}
	=\frac{k+\shift}{t} \frac{2a_k}{1-a_k^2},  \qquad k=0,\ldots,n-2,
\end{equation}
 with  boundary conditions 
$
	a_{-1}=a_{n-1}=1.
$
\end{theorem}

We will prove this theorem Section \ref{sec:proofOfMain}. 

  In \cite[Theorem 3.10]{VanAssche1}, the equivalent statement for the case $n=\infty$ was proved using a fix point argument for $t<1$. The same proof (almost per verbatim) can be applied for proving Theorem \ref{thm:last}. However, for $t\geq 1$ the argument in \cite{VanAssche1} breaks down and left as an open problem. The proof that we present here works for all $t>0$. It  is based on a convexity argument showing that $H_n$ in \eqref{eq:aHamiltonian} has a unique minimizer on $[0,1]^{n-1}$ that occurs in the interior of the domain.

Now that we have established existence and uniqueness, we look at a double scaling limit  for  $ a_{n+k}$  as  $n \to \infty$ such that 
$$
t=n-2^{-\frac13} \sigma n^{\frac 13}, \text{ and } k/n^{\frac 13} \to x \in (-\infty,0).
$$
Key to our proofs is an approximation to the unique solution in terms of particular solutions to the PII equation. Note that if we set 
$
t=n
$, insert the ansatz  $a_{k}= A(k/n)$ in \eqref{eq:dPIIaa} and take the limit $n\to \infty$ it follows that 
$$
A(\xi)=\frac{\xi A(\xi)}{1-A(\xi)^2}.
$$
Together with the boundary condition $A(0)=1$, this gives $A(\xi)=\sqrt{1-\xi}$.  Note that $A(\xi)$ vanishes as a square root near $\xi=1$ and it turns out that $A(k/n)$ is not a good approximation for $a_k$ for $k \approx n$. In this regime we will use the ansatz \eqref{eq:ansatz} with $\y(x)=2^{\frac 13}\yy(2^{\frac 13} x+\sigma)$ and $\yy$ a solution to the PII equation
\begin{equation}
	\label{eq:PII2a}
	\yy''(x) = x\yy(x) + 2 \yy^3(x).
\end{equation}
The solution that $\yy$ should  to match $A(\xi)$ and this means that we look for solutions with the following asymptotic behavior at $-\infty$:
\begin{equation}
	\label{eq:boundary1pre}
	\yy (x) = \sqrt{-x/2}\left(1 + o(1)\right), \quad x\to -\infty.
\end{equation}
The other condition on the solution $\yy(x)$ comes from the boundary condition $a_{n-1}=1$. Indeed, because of the factor $n^{-1/3}$ the approximation \eqref{eq:ansatz} can only satisfy the condition if $\yy(x)$ has a pole at $\sigma$. 

Solutions to the PII equation \eqref{eq:PII2a} are meromorphic for $x\in \mathbb C$. Their behavior near $x\to \infty$ is subject to the Stokes phenomenon: the asymptotics is different in the six  sectors of the form $k \pi /3\arg x <(k+1)\pi/3$ for $k=0,1,\ldots,5.$. The rays separating these sectors  are called critical rays  or Stokes lines.  Note that the negative real line is such a critical line. As was shown by Boutroux \cite{Boutroux}, there are special solutions  that have  the same  asymptotic behavior in two  adjacent sectors. Such solutions are often referred to as tronqu\'ee solutions. In fact, Boutroux showed that there exists a one-parameter  family of solutions to \eqref{eq:PII2} that satisfy \eqref{eq:boundary1pre} for $|x| \to \infty$ such that $|\arg(-x)|< \pi/3$. Moreover, the $o(1)$ term in \eqref{eq:boundary1pre}  can be made more precise (see equation (11.5.32) in \cite{FIKN}) and we have
\begin{equation}
	\label{eq:boundary1}
	\yy (x) = \sqrt{-x/2}\left(1 -\tfrac18 (-x)^{-3}+\mathcal O((-x)^{-6})\right), 
\end{equation}
as $|x| \to \infty$ such that $|\arg(-x)|< \pi/3$. The most famous example the family of solutions is the Hastings-McLeod solution \cite{HM}, which we will denote by $\yy_{0}(x)$. This solution is a real analytic solution defined for all $x \in \mathbb R$, with asymptotic behavior 
\begin{equation}
	\label{eq:HMboundary}
	\yy_{0}(x) =\Ai(x), \qquad x \to + \infty,
\end{equation}
where $\Ai(x)$ is the Airy function. The other (real) solutions to \eqref{eq:PII2} that satisfy \eqref{eq:boundary1} are parameterized by $b \in \mathbb R$ and characterized by the following asymptotic behavior  $x\to-\infty$
\begin{equation} \label{eq:familybatinfinity}
	\yy_b(x)=\yy_0(x)+ \frac{b}{(-x)^\frac14} \exp\left(-\frac{2 \sqrt2 }{3} (-x)^{3/2}\right)(1+o(1)), \qquad x \to -\infty.
\end{equation}
This means that the solutions $\yy_b$ only differ by superexponentially small terms as $x\to-\infty$.

\begin{remark} \label{remark:Stokes} The isomonodromy approach to Painlev\'e equations (see \cite{FIKN} for a thorough discussion) provides  a bijection from the set of Stokes parameters
	$$
	\{(s_1,s_2,s_3)\in \mathbb C^3 \mid s_1-s_2+s_3+s_1s_2s_3=0\},
	$$
	to the set of solutions of the PII equation. By choosing 
	$s_1=-i$ and $s_3=i,$
	the parameter $s_2$ can be arbitary and this is precisely the one-parameter family of interest to us. By choosing $s_2\in \mathbb R$ we obtain solutions that take real values for $x \in \mathbb R$.   We refer to Chapter 10 in \cite{FIKN}  (and the references therein) for a discussion on asymptotics on the critical rays and Chapter 11 for a discussion on the Stokes phenomenon.  See also \cite{Kapaev,Novok}. Then  \eqref{eq:familybatinfinity} is part 3) of Theorem 10.2 in \cite{FIKN}, which also  states $b$ is proportional to $s_2$ as follows:
	$
	b=-\frac{s_2}{\sqrt{\pi} 2^{7/4}}. 
	$
	
\end{remark}

 
The solutions $\yy_b$ for $b\neq 0$ will be meromorphic, possibly with poles on the real axis. It is also known  that the location of the poles depends meromorphically on $b$. From the boundary condition in \eqref{eq:boundary1} it follows that there must either be a first pole on the real line (counted from $-\infty$), or no poles. We will denote the first pole by $\sigma(b),$ with $\sigma(b) =+\infty$ if  $\yy_b$ has no poles. The Hastings-McLeod solution has no poles on the real line and thus $\sigma(0)=+\infty$. In fact, the Hastings-McLeod solution is special since it is a bitronquee solution: it has a uniform asymptotics in the two sectors defined by $|\arg(-x)|<\pi/3$ and $|\arg x|<\pi/3$ and each of them contains a critical ray. It is the unique such solution among the family $\{\yy_b\}$ and the other (real) solutions do have poles on the real line.  This can be seen from the following asymptotic formula (see \cite[Theorem 10.1]{FIKN} and Remark \ref{remark:Stokes} above)
\begin{equation}\label{eq:stokesplusinf}
	\yy_b(x)= \mathrm{sgn}(b)\sqrt{x/2} \cot \left(\frac{\sqrt{2}}{3} x^{3/2}+\frac{\gamma}{2}\log (8\sqrt 2x^{3/2})+\eta+\mathcal O(x^{-3/2})))\right)+\mathcal O(x^{-1}), 
\end{equation}
as  $x\to +\infty$, where $\gamma$ and $\eta$ are constants that depend explicitly on $b$.  This expansion is valid uniformly outside neighborhoods of the poles of the leading term. From this expansion (and the fact that $\yy_b$ is real valued) we see indeed that $\yy_b$ has infinitely many poles on the real line and  we can even find the asymptotic distribution as $x\to + \infty$. In particular, we conclude that $\sigma(b)< + \infty$ iff $b\neq 0$.

Given $\sigma\in \mathbb R$, we will be interested in finding the solution $\yy_b$ such that $\sigma(b)=\sigma.$ Here we encounter a major challenge, as it is not obvious at all that  such a solution  exists.   By using a power series expansion one can verify that there exists a one-parameter family of solutions to the PII with a pole at any given $\sigma\in \mathbb R$. It follows from the form of the equation that solutions to  PII can only have poles of order one with residue $\pm 1$ (and we are interested in solutions with residue $-1$). However, there is no simple way of deriving the asymptotic behavior of these solutions for $x \to -\infty$ and  it is very hard to tune the remaining parameter such that the solution also satisfies \eqref{eq:boundary1}.  Indeed, the isomonodromy approach \cite{FIKN} makes it possible to relate asymptotic behavior at different singularities, but any $\sigma \in \mathbb R$ is a regular point for PII equation and it is far from obvious to relate the asymptotic behavior at $x\to -\infty$ to the behavior near the pole at $x=\sigma$. This has only be done in case for special symmetric solutions \cite{Kitaev}, but these have an asymptotic behavior that is different from \eqref{eq:boundary1} and thus not important to us.

In our next main is, we show that solutions $\yy_b$ and the pole location $\sigma(b)$  are monotone in the parameter $b$.  More precisely, we show the following:
\begin{theorem}
	\label{thm:ybounds}
	For $b>0$ there exists a exists an $\sigma(b)\in \mathbb{R}$ such that the solution $\yy_b(x)$ is analytic on the interval $(-\infty, \sigma(b))$ and has a simple pole at $\sigma(b)$ with residue $-1$. Moreover
	\begin{enumerate}
		\item The map $b\mapsto \sigma(b)$ is a strictly decreasing bijection from $(0,\infty)$ to $\mathbb{R}$.
		
		\item If $b >b' \ge 0$, then $\yy_b(x) > \yy_{b'}(x)$ for all $x\in (-\infty, \sigma(b))$.	
\end{enumerate}
\end{theorem}
To the best of our knowledge these results are new. In our proof, the map $b \mapsto \sigma(b)$  is not explicit and it would be very interesting if one could get better control on it.  

An important ingredient in the proof of Theorem \ref{thm:ybounds} is the following inequality.
\begin{theorem}	
\label{thm:ineq}	
	Let $b \geq 0$. Then we have
	\begin{equation} \label{eq:mainineq}
	\yy_b(x) > \max \left(\sqrt{-x/2}\sqrt{\frac{\sqrt{1-2x^3}-1}{\sqrt{1-2x^3}+1}},\sqrt{-x/6}\right),
	\end{equation}
	for  $x \leq \min(\sigma(b),0)$.
	
\end{theorem}
The proofs of Theorem \ref{thm:ybounds} and \ref{thm:ineq} are given in Section \ref{sec:cont}.

Note that in this theorem we allow $b=0$ and thus this inequality is valid for the Hastings-McLeod solution and we are not aware of a similar type of lower bound in the literature.  Roughly speaking, we prove  Theorem \ref{thm:ineq} by showing that any solution $\yy$ to the PII equation crosses the graphs of one the two function at the right-hand side of \eqref{eq:mainineq}  from above, will be decreasing near the point of intersection and does not have time to gather enough momentum to bend back and remain positive. 

We now come to the main result which characterizes the limit of $a_{n+k}$ in the double scaling regime. 

\begin{figure}[t]
	\begin{center}
	\includegraphics[scale=.7]{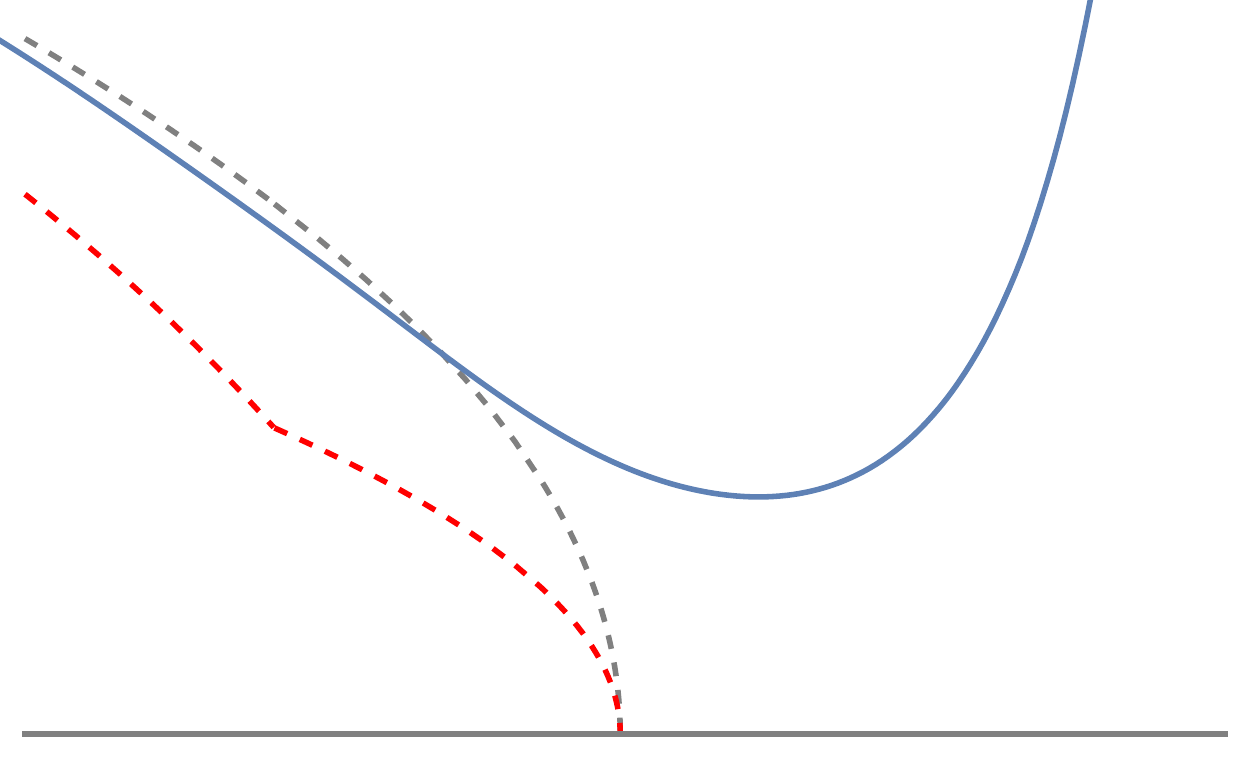}
	\end{center}
\caption{The solid line is a numerical approximation to a solution $\yy_b$ with a pole on the positive axis and asymptotic behavior $\sqrt{-x/2}$ as $x \to -\infty$. The dashed lines are the lower bound \eqref{eq:mainineq} and the curve $y=\sqrt{-x/2}$.}
	\end{figure}
\begin{theorem}
\label{thm:main}
Let $a_k \in (0,1)$ for $k=0,\ldots,n-2$ be the unique solution to \eqref{eq:dPII} with boundary conditions \eqref{eq:BC}. For $\sigma \in \mathbb R$ and set
$$
t=n- 2^{-\frac13} \sigma n^{1/3}.
$$
Then, for $x<0$, we have
	$$
	\lim_{n\to \infty, k/n^{1/3} \to x}n^{1/3} a_{n+k}=2^{\frac13}\yy_{b(\sigma)}(2^{\frac13} x+ \sigma)
	$$
where $\yy_{b(\sigma)}$ is the solution of the PII equation uniquely characterized by   \eqref{eq:boundary1pre} and having its first real pole at $x=\sigma$.
\end{theorem}

The proof of this theorem is the topic of Section \ref{sec:proofOfMain}.

The characterization of the solution to \eqref{eq:dPII} and \eqref{eq:BC}  as  the unique minimizer  (and unique critical point) of the Hamiltonian $H$  \eqref{eq:aHamiltonian} will be central in our proof of Theorem \ref{thm:main}.  We will construct an approximation to the unique solution to the d-PII equation using  the solution $\yy_{b(\omega)}$ for the continuous PII equation. To prove that this is indeed a good approximation we will use the Kantorovich Theorem for Newtons method for finding the zero of $\nabla H$.  More precisely, we will change $\nabla H_n$ slightly and work with a function $F$ that has the same zero defined by $(F(a))_k=(1-a_k^2)(\nabla H_n(a))_k$. Newton's method for finding a zero of $F$ is the iteration
$$
a^{(k+1)}= a^{(k)}-\left(F'(a^{(k)})\right)^{-1}F(a^{(k)}), \qquad k=0,1,2,\ldots
$$
starting from a first guess $a^{(0)}$.  A key difficulty in this approach is that $F'$ can have very small eigenvalues (mainly due to the behavior in  the Painlev\'e regime) and this causes the norm  of $(F')^{-1}$ to blow up as $n \to \infty$.   To counter this effect, we will start with a very good first approximation $a^{(0)}$. To show that this is indeed good enough requires  sharp estimates on the value of $F$ at the minimizer and $F'$ from above and below in the neighborhood of the minimizer.  The estimate 
$$\yy_b(x) \geq \sqrt{-x/6},$$
from Theorem \ref{thm:ineq} is crucial in the estimating $\|(F')^{-1}\|_\infty$ from below.

\section{Proof Theorems \ref{thm:ybounds} and \ref{thm:ineq}} \label{sec:cont}
In this section we prove Theorems \ref{thm:ybounds} and \ref{thm:ineq}. 

Instead of working with the standard form of the PII equation in \eqref{eq:PII2a} we will use the form for $\y_b(x)=2^{\frac{1}{3}}\yy_b(2^{\frac{1}{3}}x)$, which satisfies
\begin{equation} \label{eq:PIIalt}
	\y_b''(x)=2x \y_b(x)+2 \y_b(x)^3
\end{equation}
with 
\begin{equation} \label{eq:familybatinfinityu}
	\y_b(x)=\y_0(x)+ 2^{\frac{1}{3}}\frac{b}{(-x)^\frac14} \exp\left(-\frac{4}{3} (-x)^{3/2}\right)(1+o(1)), \qquad x \to -\infty
\end{equation}
where $\y_0(x) =2^{\frac{1}{3}} \yy_0(2^{\frac{1}{3}}x)$ is the Hastings-McLeod solution for \eqref{eq:PIIalt}. With this notation change we can observe that $\omega(b)$ the pole of $u_b$ will be at  $x = 2^{-\frac{1}{3}}\sigma(b),$ where $\sigma(b)$ is the pole of $\yy_b$.

We make the switch to $\y_b$ in order to be consistent with the proof of the main result, Theorem \ref{thm:main}, which is quite technical, and clearer using the $\y_b$ notation. We begin by reformulating Theorems  \ref{thm:ybounds} and  \ref{thm:ineq} for $\y_b$. 

\begin{theorem}
	\label{thm:ubounds}
	For $b>0$ there exists a exists an $\omega(b)\in \mathbb{R}$ such that the solution $\y_b(x)$ is analytic on the interval $(-\infty, \omega(b))$ and has a simple pole at $\omega(b)$ with residue $-1$. Moreover
	\begin{enumerate}
		\item The map $b\mapsto \omega(b)$ is a strictly decreasing bijection from $(0,\infty)$ to $\mathbb{R}$.
		
		\item If $b >b' \ge 0$, then $\y_b(x) >\y_{b'}(x)$ for all $x\in (-\infty, \omega(b))$.	
	\end{enumerate}
\end{theorem}

\begin{theorem}
	\label{thm:inequ}
	Let $b \geq 0$. Then we have
	\begin{equation} \label{eq:maininequ}
		\y_b(x) > \max \left(\sqrt{-x}\sqrt{\frac{\sqrt{1-2x^3}-1}{\sqrt{1-2x^3}+1}},\sqrt{-x/3}\right),
	\end{equation}
	for  $x \leq \min(\omega(b),0)$.
\end{theorem}

Proving these is equivalent to proving Theorems \ref{thm:ybounds} and  \ref{thm:ineq}.

The proof will follow the following steps:
\begin{enumerate}
	\item Prove several preliminary results on the behavior of $\y_b$.
	\item Prove that $\y_b(x) \ge \sqrt{-x}\sqrt{\frac{\sqrt{1-2x^3}-1}{\sqrt{1-2x^3}+1}}$ for $x< \min(\omega(b),0)$.
	\item Prove that Theorem \ref{thm:inequ} holds for the Hastings-McLeod solution $\y_0$.
	\item Prove Theorem \ref{thm:ubounds}, which in turn implies that Theorem \ref{thm:inequ} holds for all $b> 0$.
\end{enumerate}

\subsection{Some preliminaries}

When studying the behavior of $\y_b$ it is natural to  divide the plane in four regions separated by the curves $y=0$ and $x+y^2=0$. Indeed, if $(x,\y_b(x))$ is a point on one of these curves then $u_b''(x)=0$. If $(x,\y_b(x))$ is in region I or  III then $\y_b''(x)>0$ and $\y_b''(x)<0$ in regions II and IV. See also Figure \ref{fig:fourregions}.

\begin{figure}[t]
	\begin{center}
		\begin{tikzpicture}[domain=0:2, scale = 1]
			\draw[black, line width = 0.50mm]   plot[smooth,domain=-2:2] ( {-(\x)^2},\x);
			\draw[black, line width = 0.50mm]   (-4,0)--(4,0);
			\draw (-3,1) node {I};
			\draw (1,1.5) node {II};
			\draw (-3,-1) node {IV};
			\draw (1,-1.5) node {III};
		\end{tikzpicture}
		\caption{The curves $y^2+x=0$ and $y=0$  divide the plane in four regions. Any solution  $\y$ to \eqref{eq:PIIalt} is convex when $(x,\y(x))$ in region $II$ and $IV$, and concave when it is in $I$ and $III$. } \label{fig:fourregions}
	\end{center}
\end{figure}
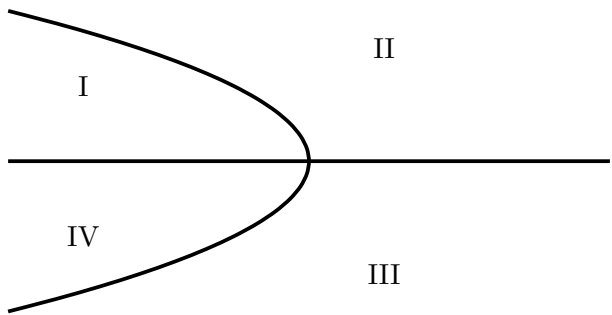

We start can now observe that after a solution $\y_b$  enters region II then it will never enter region I again.  

\begin{figure}[t]
	\begin{center}
		\begin{tikzpicture}
			\begin{axis}[
				axis equal image,
				xmin = -3,
				xmax = 0.5,
				ymin = -0.6,
				ymax = 2.5,
				inner axis line style={<->, thick},
				xlabel={\large $x$},
				ylabel={\large $y$},
				axis lines=middle,
				ticks = none,
				]
				\addplot[dashed, blue, domain = -3:0, samples = 50]  {sqrt(-x)} node[pos=0.1, below] {$\sqrt{-x}$};
				\addplot[ thick, black, domain = -3:0] {-(1/2)*(x+1)+1} node[pos=0.1, above] {$k$};
				\draw[very thick, teal, <->] (axis cs: -2.1,2.2) .. controls (axis cs: -2,2) and (axis cs: -1.3, 1.2).. (axis cs: -1, 1) .. controls (axis cs: -0.7, 0.75) and (axis cs: -0.5, 0.6) .. (axis cs:-0.2,0.2) node[pos = 0.9, left] {$\y_b$};
				\node[label={0:{\footnotesize $ (x_0, \sqrt{-x_0})$}},circle,fill,inner sep=2pt] at (axis cs:-1,1) {};		
			\end{axis}
		\end{tikzpicture}
		\caption{The dashed curve represents $y=\sqrt{-x}$ and the solid straight line is  part of the tangent line to -$y=\sqrt{-x}$ at the point where the solution $\y_b$ crosses (or touches) the curve $y=\sqrt{-x}$ from above. Since any solution of the PII equation is convex in region $II$, the graph of $\y_b$ left of the intersection point remains above the tangent line (until it hits a pole). Since this is in contradiction with the asymptotic behavior \eqref{eq:boundary1}, no $\y_b$ can enter region $I$ from region $II$. } \label{fig:lem1}
	\end{center}
\end{figure}
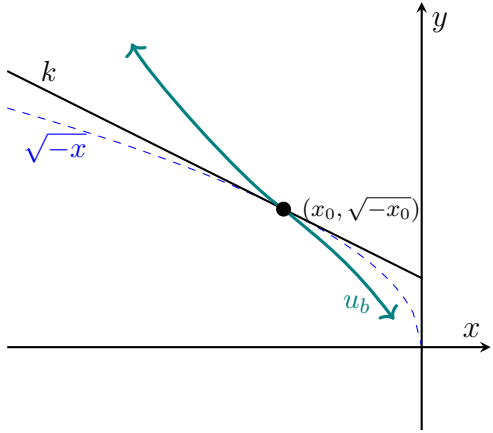
\begin{lemma}
	If $\y_b(x_0)>\sqrt{-x_0},$ for some $x_0 \in (-\infty, \omega(b))$,  then $\y_b(x) > \sqrt{-x}$ for all $x \in (x_0,\min (\omega(b),0).$
\end{lemma}
\begin{proof}
	Assume that $\y_b(x)$ does cross, or touches, the curve $y = \sqrt{-x}$ at some point $\xi > x_0$. That is, there exists a $\xi > x_0$ such that $\y_b(\xi) = \sqrt{-\xi}$ (and we take $\xi$ to be the first time after $x_0$ this happens).  It follows that we must have $\y_b'(\xi) \le -\frac{1}{2 \sqrt{-\xi}}$ since the slope must be at least as steep as the slope of $y = \sqrt{-x}$. We will show that for $x\in (-\infty,\xi)$, the solution $\y_b(x)$ will be above the line 
	\[
	k: y = \y'_b(\xi)(x-\xi) + \sqrt{-\xi} 
	\]
	which passes through the point $(\xi, \sqrt{-\xi})$ with slope $\y'_b(\xi)$. Since $\y_b'(\xi) \le -\frac{1}{2 \sqrt{-\xi}}$ any such solution $\y_b(x)$ will not satisfy the boundary condition given in \eqref{eq:boundary1}, giving us a contradiction.
	
	Notice that if $\y_b(x) > \sqrt{-x},$ then from \eqref{eq:PIIalt} it follows that $\y_b''(x)>0.$ This means that $\y_b'(x)$ is increasing on $(x_0,\xi),$ so $\y_b(x)$ must be above the line that passes through the point $(\xi, \sqrt{-\xi})$ with slope $\y_b'(\xi)$. This can in fact be extended to the entire interval $(-\infty,\xi)$ by observing that if $\y_b(x)$ is above this line it is also above $\sqrt{-x},$ meaning that $\y_b''(x)>0$. Therefore it follows that once a solution $\y_b(x)$ to \eqref{eq:PII2} with boundary condition \eqref{eq:boundary1} is above the curve $y=\sqrt{-x}$ it will remain above this curve. 
	
	See also Figure \ref{fig:lem1} for an illustration.
\end{proof}

A crucial point in our methods for proving Theorems \ref{thm:ubounds} and \ref{thm:inequ}  is the following representation for $\y_b$:

\begin{proposition}
	\label{prop:ysatisfies}
	For $b\in \mathbb{R}$ the function $\y_b(x)$ satisfies 
	\begin{equation}\label{eq:integratedform}
		\big(\y_b'(x)\big)^2 + 2\int_{-\infty}^x \big( \y_b^2(s) + s\big) ds = \big( \y_b^2(x) + x\big)^2
	\end{equation}
	for $x\in (-\infty, \omega(b))$.
\end{proposition}

\begin{proof}
	Start with \eqref{eq:PII2} and multiply it with $2\y'_b$:
	$$
	2 \y_b'(x)	\y''_b (x)= 4 \y'_b(x) \y_b(x)^3+4x  \y_b'(x)\y_b(x).
	$$
	Now add $2\y_b(x)^2+2x$ to both sides:
	$$
	2 \y_b'(x)	\y''_b (x)+2 \y_b(x)^2+2x= 4 \y'_b(x) \y_b(x)^3+4x  \y_b'(x)y_b(x) +2\y_b(x)^2+2x.
	$$
	Observe that because of the asymptotic behavior at $ -\infty$, we in fact have that $\y_b(x)^2+x$ is integrable at $-\infty$ and thus we can integrate
	$$
	(\y_b'(x))^2+2 \int_{-\infty}^x (\y_b(s)^2+s) ds= (\y_b(x))^4+2x \y_b(x)^2+x^2+c
	$$
	where $c$ is an integration constant. In fact, by taking the limit $x\to -\infty$ at both sides we find $c=0$ and this proves the statement. 
\end{proof}

From the asymptotic behavior for  $x \to -\infty$  we see that $\y_b(x)$ starts in region $I$. Eventually, the solution $\y_b(x)$ will leave this region (at the latest at $x=0$). Let $x_*(b)$ be the first time that $\y_b(x)$ leaves this region.

\begin{lemma}
	\label{cor:yPrimeBnd}
	Let $x_*(b)$ to be the first time that $\y_b(x)$ intersects $y=0$ or $y^2+x=0$, then
	\[
	\y'_b < \y_b^2(x)+ x \quad \text{ for } x\in (-\infty, x_*(b)). 
	\]
\end{lemma}
\begin{proof}
	First note that $\y_n''(x)<0$ for $x<x_*(b)$. This implies that the integrand at the left-hand side of \eqref{eq:integratedform} in Proposition \ref{prop:ysatisfies} is negative and thus
	\[
	\big(\y_b'(x)\big)^2 > \big( \y_b^2(x) + x\big)^2.
	\]
	We can see from this that $\y_b'(x)$ has a fixed sign on $(-\infty,x_*(b))$. It follows from the asymptotic behavior of $\y_b$ at $-\infty$ that $\y_b'(x)$ must in fact be negative for $x$ large and negative, and $\y_b^2(x) + x<0$ for $x< x_*(b)$. Taking appropriate square roots gives us the final result.
\end{proof}
The latter result can be used to find a lower bound for solutions $\y_b$ that are positive.

\subsection{A first lower bound for $\y_b$}
In this section we show one half of the inequality in Theorem \ref{thm:inequ}. In particular we show that $\y_b(x)$ is bounded below by the first function listed in Theorem \ref{thm:inequ}. 
\label{subsec:lowerbound1}

\begin{theorem}
	\label{thm:lowerBnd1}
	For and $b\in \mathbb{R}$ such that $\y_b$ is positive for $-\infty<x<\min(\omega(b),0)$ we have that 
	\[
	\y_b(x) \ge \sqrt{-x} \rootroot
	\]
\end{theorem}
\begin{proof}
	We will use the notation
	\begin{equation}
		\label{eq:ex}
		e(x) = 1-\frac{\sqrt{1-4x^3}-1}{\sqrt{1-4x^3}+1} = \frac{2}{1+\sqrt{1-4x^3}}.
	\end{equation}
	With this notation we can rewrite the inequality as
	\[
	\y_b(x) > \sqrt{-x(1-e(x))}.
	\]
	Note that from the asymptotic behavior of $e(x)$ and $\y_b(x)$ for $x\to -\infty$ we find that the inequality holds  for $x$ large and negative.
	
	Suppose now that the inequality does not hold for all $x<\min(\omega(b),0)$. That means there exists a point $x_0<\min(\omega(b),0)$ where $\y_b(x)$ crosses the curve $y = \sqrt{-x(1-e(x))}$. From Lemma \ref{cor:yPrimeBnd} we get that $\y_b'(x_0) < e(x_0) x_0$. This means that $\y_b(x)$ even crosses the line with slope $e(x_0)x_0$ that passes through the point $(x_0, \sqrt{-x_0(1-e(x_0))})$ at $x_0$. This line is given by:
	\[
	\ell: y= \sqrt{-x_0(1-e(x_0))}+e(x_0) x_0 (x-x_0).
	\]
	It can be checked that this line passes through the origin. Since $\y_b''(x)<0$ for $\y_b(x)$ in the region enclosed by $y =\sqrt{-x}$ and the negative real axis, this is true even for $\y_b(x)$ between $\ell$ and the negative real axis. Since $\y_b'(x)< e(x_0) x_0$ on this region it follows that $\y_b(x)$ will leave the region at a point on the negative real axis. This is a contradiction to the assumption that $\y_b(x)>0$ for all $x< \min(0,\omega(b))$, and thus no such $x_0$ can exist, which gives the desired inequality.
\end{proof}

We will prove in Section \ref{subsec:monotonicity} that $\y_b(x)> \y_0(x) >0$ for all $b\ge 0$ and $x<\min(0,\omega(b))$. This will give us that the inequality holds for all $\y_b$ with $b\ge 0$. Before we come to that we first recall some properties of the Hastings-McLeod solution $b=0$ and  add some new results.

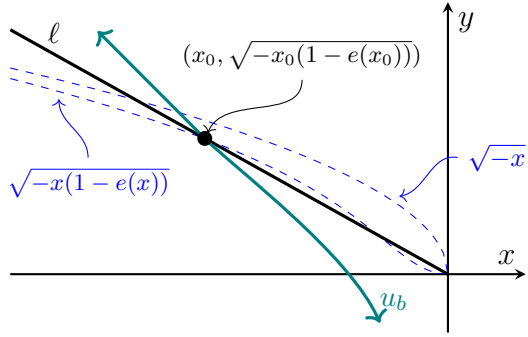
\begin{figure}
	\begin{center}
		
		\begin{tikzpicture}
			\begin{axis}[
				axis equal image,
				xmin = -4.5,
				xmax = 0.8,
				ymin = -0.6,
				ymax = 2.8,
				inner axis line style={<->, thick},
				xlabel={\large $x$},
				ylabel={\large $y$},
				axis lines=middle,
				ticks = none,
				]
				\addplot[dashed, blue, domain = -4.5:0, samples = 100]  {sqrt(-x)} ;
				\addplot[dashed, blue, domain = -4.5:0, samples = 100]  {(-x*(sqrt(1-4*x^3)-1)/(sqrt(1-4*x^3)+1))^(1/2)} ;
				\draw[blue, ->, out=90,in=-90,] (axis cs: -3.7, 1.2)  coordinate (a) to (axis cs: -4,1.8);
				\node[blue, below] at (a) {\footnotesize $\sqrt{-x(1-e(x))}$};		
				\draw[blue, ->, out=170,in=10,] (axis cs: 0.1, 1.2)  coordinate (b) to (axis cs: -0.5,0.8);
				\node[blue, right] at (b) {\small $\sqrt{-x}$};		
				\addplot[ very thick, black, domain = -4.5:0] {-0.559*x} node[pos=0.1, above] {$\ell$};
				\draw[very thick, teal, <->] (axis cs: -3.6,2.5) .. controls (axis cs: -3.5,2.4) and (axis cs: -2.8, 1.697).. (axis cs: -2.5, 1.397) .. controls (axis cs: -2.2, 1.097) and (axis cs: -1, 0.2) .. (axis cs:-0.7,-0.5) node[pos = 0.9, right] {$\y_b$};
				\draw[ ->, out=-90,in=70,] (axis cs: -1.5, 2)  coordinate (c) to (axis cs:-2.47,1.45) coordinate (d);
				\node[above] at (c) {\footnotesize $ (x_0, \sqrt{-x_0(1-e(x_0))})$};			
				\node[circle,fill,inner sep=2pt] at  (axis cs:-2.5,1.397) {};		
			\end{axis}
		\end{tikzpicture}

		\caption{The top dashed curve represent $\y=\sqrt{-x}$ and separates region $I$ from region $II$. The lower dashed curve represents  $x\mapsto \sqrt{-x}\rootroot$. As shown in the proof of Theorem \ref{thm:lowerBnd1}, once a solution $\y_b$ crosses the lower dashed curve it has to remain below the line through the intersection point and the origin. It will thus take negative values. } \label{fig:lem2}
	\end{center}
\end{figure}

\subsection{The Hastings-McLeod solution}
\label{subsec:HM}

In this section we show the second bound in Theorem \ref{thm:inequ} for the Hastings-McLeod solution $\y_0(x)$.

The existence and uniqueness of the Hastings-McLeod solution $\yy_0(x)$ to the Painlev\'e equation \eqref{eq:PII2a} was already established by Hastings and McLeod \cite{HM}. This is the solution that satisfies boundary conditions \eqref{eq:boundary1} and \eqref{eq:HMboundary}.  The transformation $\y_0(x) = 2^{\frac{1}{3}}\yy_0(2^{\frac{1}{3}}x)$ gives us existence and uniqueness for  $\y_0$. They also prove the following:

\begin{lemma} The Hastings-McLeod solution $\y_0$ has the following properties
	\begin{enumerate}
		\item 	$\omega(0)=+ \infty$, ie $\y_0$ has no poles on the real line. 
		\item 	The solution $\y_0(x)$ crosses the curve $y=\sqrt{-x}$ exactly once. 
		\item 	The solution $\y_0(x)$ is always positive.
		\item 	The solution $\y_0(x)$ is strictly decreasing. 	
	\end{enumerate}
\end{lemma}
\begin{proof} These statements are parts of Theorem 1.3 in  \cite{HM}.
\end{proof}
Using these properties together with the lower bound proved in Theorem \ref{thm:lowerBnd1} we can prove that $\y_0(x) > \sqrt{-x/3}$, which will show that Theorem \ref{thm:inequ} holds for $b=0$. 
\begin{lemma} \label{lem:keyineq}
	We have $\y_0(x) >\sqrt{-x/3}$ for $x <0.$
\end{lemma}
\begin{proof}
	Since we know that $\y_0$ is strictly positive, the inequality in Theorem \ref{thm:lowerBnd1} applies. In the region $x<-(\frac34)^{1/3}$ this is a stronger inequality than the one we want to prove. It thus  remains to prove the inequality for $y_0(x) > \sqrt{-x/3}$ for $x \in (-(\frac{3}{4})^{1/3},0).$ 
	
	Suppose that $\y_0(x)$ touches or crosses the curve  $y=\sqrt{-x/3}$ at $x=x_0$. Then for some time it will remain under the line $m$ defined by 
	$$
	m: y=\sqrt{-x_0/3}-\frac{x-x_0}{2 \sqrt{-3 x_0}}.
	$$
	In fact, it will remain under that line at least until $m$ crosses the curve $y=\sqrt{-x}$, since $\y_0''(x)<0$ in this region. A simple computation shows that $m$ and  $y=\sqrt{-x}$ cross at $(x_1,\sqrt{-x_1})$ where 
	$$x_1=(5-2 \sqrt 6 )x_0 > x_0/9.$$
	Let $\y_*$ be the solution to \eqref{eq:PIIalt} that passes through the point $(x_1,\sqrt{-x_1})$ and has derivative equal to the slope of the line $m$. In other words $\y_*$ satisfies $\y_*(x_1) = \sqrt{-x_1}$ and $\y_*'(x_1) = -\frac{1}{2\sqrt{-3x_0}}$.
	
	\noindent \textbf{Claim:} Let $x^- = \inf_x\{\y_*(x) < 0 \}$, then $\y_*(x) \ge  \y_0(x)$ for all $x_1<x< x^-$.
	
	Notice that $\y_*(x_1) >  \y_0(x_1)$ and $\y_*'(x_1) >  \y_0'(x_1)$ since $\y_0''(x)<0$ for $x \in (x_0,x_1)$. Then since $\y_*(x)$ satisfies \eqref{eq:PIIalt} we have that $\y_*''(x)>0$ for $x>x_1$, and $0=\y_*''(x_1) \ge \y_0''(x_1)$. Moreover, we can see that for positive solutions $\y$ to \eqref{eq:PIIalt} the corresponding second derivative $\y''$ is increasing in $\y$. This means that these inequalities will propagate forward giving us $\y_*(x) >  \y_0(x),$ $\y_*'(x) >  \y_0'(x)$, and $\y_*''(x) \ge \y_0''(x)$ for all $x_1<x< x^-$. This proves the claim.
	
	We will show that $\y_*(-2x_0)<0$ which implies $\y_0(x)<0$ for some $x< -2x_0$. This is a contradiction of the know properties of $\y_0$, which means that no such point $x_0$ exists and $\y_0$ never crosses the curve $y = \sqrt{-x/3}$. 
	
	To show that $\y_*(-2x_0)<0$ we begin by noting that $\y_*$ can be written as 
	$$\y_*(x)=\y_*(x_1)+(x-x_1)\y_*'(x_1)+ \int_{x_1}^x \int_{x_1}^t \y_*(s)(\y_*(s)^2+s)ds.$$
	Now observe that 
	$$\y_*(x_1)+(-2x_0-x_1)\y_*'(x_1) < \frac{\sqrt{-x_0}}{3}-\frac{-2x_0-x_0}{2\sqrt{-3x_0}} < -\frac{\sqrt{-x_0}}{2}.$$
	Indeed, the first inequality holds since $\y_*'(x_1)=-\frac{1}{2\sqrt{-3x_0}}, \y_*(x_1) = \sqrt{-x_1}$ and $x_0<x_1<\frac{x_0}{9}$, and the second inequality can be checked by computation.
	
	Next we estimate the double integral using the fact that $\y_*(s)^2+s>0$ for $s>x_1$,
	\begin{multline*}	\int_{x_1}^x \int_{x_1}^t \y_*(s)(\y_*(s)^2+s)ds \leq  
		\int_{x_1}^x \int_{x_1}^t \y_*(x_1)(\y_*(x_1)^2+s)ds \\= \sqrt{-x_1} \int_{x_1}^x \int_{x_1}^t (s-x_1)ds
		=\frac16 \sqrt{-x_1}(x-x_1)^3 
	\end{multline*}
	By taking $x=2x_0$ and using $x_1>x_0/9$  (and $x_1<0$) we find 
	$$\int_{x_1}^x \int_{x_1}^t \y_*(s)(\y_*(s)^2+s)ds  \leq\frac16 \sqrt{-x_1}(-2x_0-x_1)^3=\frac{1}{18}\sqrt{-x_0}(2+1/9)^3(-x_0)^{3}.$$
	Since $x_0\geq-(\frac{3}{4})^{1/3}$ we thus find
	$$\int_{x_1}^x \int_{x_1}^t \y_*(s)(\y_*(s)^2+s)ds  \leq \frac{(2+1/9)^3}{24}\sqrt{-x_0}.$$
	Now since $\frac{(2+1/9)^3}{24}<\frac{1}{2}$ we see that $\y_*(2x_0)<0$ and this gives us the necessary contradiction: $\y_0(-2x_0)<0$. 
\end{proof}

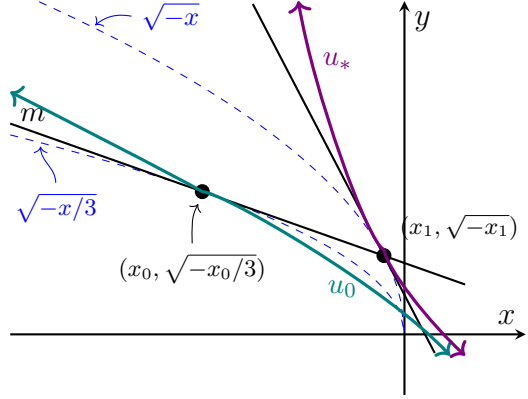
\begin{figure}[t]
	\begin{center}
		\begin{tikzpicture}
			\begin{axis}[
				axis equal image,
				xmin = -1.3,
				xmax = 0.4,
				ymin = -0.2,
				ymax = 1.1,
				inner axis line style={<->, thick},
				xlabel={\large $x$},
				ylabel={\large $y$},
				axis lines=middle,
				ticks = none,
				]
				\addplot[dashed, blue, domain = -1.3:0, samples = 100]  {sqrt(-x)} ;
				\addplot[dashed, blue, domain = -1.3:0, samples = 100]  {sqrt(-x/3)} ;
				\draw[blue, ->, out=90,in=-90,] (axis cs: -1.15, 0.5)  coordinate (a) to (axis cs: -1.2,0.62);
				\node[blue, below] at (a) {\footnotesize $\sqrt{-x/3}$};		
				\draw[blue, ->, out=170,in=40,] (axis cs: -0.9, 1.05)  coordinate (b) to (axis cs: -1.02,1.02);
				\node[blue, right] at (b) {\small $\sqrt{-x}$};		
				\addplot[thick, black, domain = -1.3:0.2] {-0.354*(x+2/3) + sqrt(2)/3} node[pos=0.05, above] {$m$};
				\addplot[thick, black, domain = -0.5:0.1] {-1.92*(x+0.0676)+0.26};
				\node[label={45:{\footnotesize $ (x_1, \sqrt{-x_1})$}},circle,fill,inner sep=2pt] at (axis cs:-0.0676,0.26) {};
				\node[circle,fill,inner sep=2pt] at (axis cs:-2/3, 0.471) {};			
				\draw[very thick, teal, <->] (axis cs: -1.3,0.8) .. controls (axis cs: -1.2,0.75) and (axis cs: -0.75, 0.52).. (axis cs: -2/3, 0.471) .. controls (axis cs: -0.6, 0.47) and (axis cs: -0.05, 0.15) .. (axis cs: 0.15,-0.07) node[pos = 0.6, below] {$\y_0$};
				\draw[very thick, violet, <->] (axis cs: -0.35, 1.1) .. controls (axis cs: -0.3,0.8) and (axis cs: -0.12, 0.3)..(axis cs:-0.0676,0.28) node[pos = 0.2, right] {$\y_*$}  .. controls (axis cs: -0.06, 0.22) and (axis cs: 0.08, 0.04) .. (axis cs: 0.2,-0.07) ;
				\draw[ ->, out=90,in=-120,] (axis cs: -0.7, 0.3)  coordinate (c) to (axis cs:-0.68, 0.43) coordinate (d);
				\node[below] at (c) {\footnotesize $ (x_0, \sqrt{-x_0/3})$};			
				\node[circle,fill,inner sep=2pt] at  (axis cs:-2.5,1.397) {};		
			\end{axis}
		\end{tikzpicture}
		
		\caption{The top dashed curve represent $\y=\sqrt{-x}$ and separates region $I$ from region $II$. The lower dashed curve represents  $x\mapsto \sqrt{-x/3}$. The orange curve is part of the graph of $\y_0$. If it intersects the lower dashed curve at a point $(x_0,\sqrt{-x_0/3})$, then it has to remain below the graph of $\y_*$ which is the solution to the PII equation that has initial condition $\y_*(x_1)=\sqrt{-x_1 }$ and $\y_*'(x_1)=1/2\sqrt{-x_1}$ with  $(x_1,\sqrt{-x_1 })$ being the intersection point of the curve $y=\sqrt{-x}$ and the tangent line to  $y=\sqrt{-x/3}$ at the point $(x_0,\sqrt{-x_0/3})$. If $x_0 >-(\frac34)^{1/3}$, then  $\y_*$ becomes negative, also $\y_0$ has to become negative and this leads to a contradiction. In other words, if a solution $\y$ to the PII equation crosses the line $y=\sqrt{-x/3}$ from above, then it has not enough momentum to bend back and remain positive.}\label{fig:lem3}
	\end{center}
\end{figure}


\subsection{Monotonicity}
\label{subsec:monotonicity}

We finish with the proof of Theorem \ref{thm:ubounds}, which together with Lemma \ref{lem:keyineq} and Theorem \ref{thm:lowerBnd1} proves Theorem \ref{thm:inequ}. 

We now prove the following. 
\begin{theorem} \label{thm:monotonicity} 
	If $b_1  > b_2 \geq 0 $ then $\omega(b_1)< \omega(b_2)$ and $\y_{b_1}(x) >\y_{b_2}(x) > \sqrt{-x/3}$ for all $x \in (-\infty, \omega(b_1))$.
\end{theorem}
\begin{proof}
	
	Let $b_1>b_2\geq 0$. From the asymptotic behavior \eqref{eq:familybatinfinity} we see that $\y_{b_1}(x_0)>\y_{b_2}(x_0)$ and $\y_{b_1}'(x_0)>\y_{b_2}'(x_0)$ for negative $x_0$  large enough  (note that $\y_b$ is meromorphic and we can differentiate the asymptotic expansion termwise).  What we need to prove is that it stretches all the way to $\omega(b_1)$.
	
	By integrating \eqref{eq:PII2} we find 
	$$\y_b(x)=\y_b(x_0)+(x-x_0) \y_b'(x_0)+ \int_{x_0}^x \int_{x_0}^t \y_b(s)(\y_b(s)^2+s) ds dt.$$
	And thus, by taking differences, we have 
	\begin{multline} \label{eq_b1minb2} \y_{b_1}(x)-\y_{b_2}(x)=\y_{b_1}(x_0)-\y_{b_2}(x_0)+(x-x_0) \left(\y_{b_1}'(x_0)-\y_{b_2}'(x_0)\right)\\+ \int_{x_0}^x \int_{x_0}^t \left(\y_{b_1}(s)-\y_{b_2}(s)\right)(\y_{b_1}(s)^2+\y_{b_1}(s) \y_{b_2}(s)+\y_{b_2}(s)^2+s) ds dt
	\end{multline}
	Suppose that $\y_{b_1}(x)\leq \y_{b_2}(x)$ for some $x < \min(\omega(b_1),\omega(b_2))$ and let  $x_1$ be the first time that $\y_{b_1}(x_1)=\y_{b_2}(x_1)$.  Then $\y_{b_1}(x)\geq  \y_{b_2}(x)$ for $x_0 <x<x_1$.  By inserting $x=x_1$ in \eqref{eq_b1minb2} we see that the left-hand side vanishes. We claim that the right-hand side is strictly positive, giving a contradiction. Indeed, the sum of the terms other than the double integral are strictly positive by assumption on $x_0$. The heart of the matter is to see that even the integrand in the double integral is positive. To see this, we first treat the case $b_2=0$. Then, by Lemma \ref{lem:keyineq},   we have $\y_{b_2}(x) > \sqrt{-s/3}$  for $s \in (x_0,x_1)$ and therefore even $\y_{b_1}(s) > \sqrt{-s/3}$. That also implies that 
	$$\left(\y_{b_1}(s)-\y_{b_2}(s)\right)(\y_{b_1}(s)^2+\y_{b_1}(s) \y_{b_2}(s)+\y_{b_2}(s)^2+s)>0,$$
	for $s \in (x_0,x_1)$. Concluding, we see that all terms at the right- hand side of \eqref{eq_b1minb2} are strictly positive and we arrive at a contradiction. This means that no such $x_1$ exist and thus we have proved the inequality in the special case $b_1 > b_2=0$.  This in particular implies that $\y_b(x)>\sqrt{-x/3}$ for all $b\geq 0$ and $ x<\min(0,\omega(b))$. Now we can repeat the argument for general $b_1$ and $b_2$ and arrive at the statement.
\end{proof}
The following is now immediate.
\begin{corollary}
	The map $b \mapsto \omega (b)$ is strictly decreasing. 
\end{corollary}

\begin{lemma} \label{lem:bto0} The following limit holds:
	$\omega(b) \to + \infty$ as $b \downarrow 0$. 
\end{lemma}
\begin{proof}
	This is known. Indeed, for $b=0$ the Hastings-McLeod solution has no poles, but for $b>0$ we have poles coming in from $+ \infty$ as follows from  \eqref{eq:stokesplusinf}
\end{proof}
\begin{lemma} \label{lem:btoinfty} The following limit holds:
	$\omega(b) \to - \infty$ as $b \to + \infty$. 
\end{lemma}
\begin{proof}
	We start by recalling that $\y_b(x)$ can not converge to a solution of \eqref{eq:PIIalt} as $b \to \infty$. Indeed, as explained in Remark \ref{remark:Stokes} the set of solutions to the PII equation are parametrized by the so-called Stokes multipliers and $b$ is precisely one such multiplier. All these multipliers must be finite and thus any limit of $\y_b(x)$ as $b\to \infty$ must fall outside the set of solutions. 
	
	Since $\omega(b)$ is strictly decreasing, $\omega_*=\lim_{b \to +\infty} \omega(b)$ must exist in $\mathbb R\cup \{-\infty\}$. Assume that $\omega_* \in \mathbb R$. Then we remark that also $\y_b(x)$ defines a sequence of monotone increasing functions on $(-\infty,\omega_* )$ as $b \to + \infty$. Hence the limit 
	$$\y_\infty(x)=\lim_{b \to \infty} \y_b(x), \qquad x \in (-\infty, \omega_*),$$
	is a well-defined measurable function (possibly taking the value $\pm \infty$). Now we note that also 
	\begin{align*}
		\y_b(x)&=\y_b(x_0)+(x-x_0)\y'_b(x_0)+ \int_{x_0}^x \int_{x_0}^s \y_b''(z)dz ds\\
		&=\y_b(x_0)+(x-x_0)\y'_b(x_0)+ 2\int_{x_0}^x \int_{x_0}^s (z\y_b(z)+\y_b(z)^3) dz ds.
	\end{align*}
	for $x_0< \omega_*$. We will prove that 
	\begin{equation}
		\label{eq:ysatisfies}
		\y_\infty(x)=\y_\infty(x_0)+(x-x_0)\y_\infty'(x_0)+ 2\int_{x_0}^x \int_{x_0}^s (z\y_\infty(z)+\y_\infty (z)^3) dz ds,
	\end{equation}
	which shows that $\y_\infty(x)$ is a solution to \eqref{eq:PIIalt} and we thus arrived at a contradiction, and wwe must have $\omega_*=-\infty$ as stated. 
	
	To prove that $\y_\infty$ satisfies \eqref{eq:ysatisfies} we need to show that 
	$$ \lim_{b \to \infty} \int_{x_0}^x \int_{x_0}^s (z\y_b(z)+\y_b(z)^3) dz ds=\int_{x_0}^x \int_{x_0}^s (z\y_\infty(z)+\y_\infty (z)^3) du ds.$$
	Split the integrand as 
	$$ \y_b(z)(z+\y_b(z)^2)= \max\left(0,\y_b(z)(z+\y_b(z)^2)\right)+ \min\left(0,\y_b(z)(z+\y_b(z)^2)\right).$$
	Now  $\y_b(z)\geq 0$ and $\y_b(z)$ is strictly increasing in $b$,  and thus
	$$
	\max\left(0,\y_b(z)(z+\y_b(z)^2)\right)$$
	is strictly increasing for $b \to +\infty$. By the monotone convergence theorem we then have that 
	\begin{align*}
		\lim_{b \to \infty} \int_{x_0}^x \int_{x_0}^s \max&\left(0,\y_b(z)(z+\y_b(z)^2)\right) dz ds\\
		&=  \int_{x_0}^x \int_{x_0}^s \max\left(0,\y_\infty(z)(z+\y_\infty (z)^2)\right) dz ds,
	\end{align*}
	where the latter possibly takes the value $+ \infty$. 
	We also know that $\y_b(z) > \sqrt{-z/3}$ and thus 
	$$|\min\left(0,\y_b(z)(z+\y_b(z)^2)\right) | \leq \frac{-2z}{3 \sqrt 3}.$$ And thus by dominated convergence we find 
	\begin{align*}
		\lim_{b \to \infty} \int_{x_0}^x \int_{x_0}^s \min &\left(0,\y_b(z)(z+\y_b(z)^2)\right) dz ds\\
		&=  \int_{x_0}^x \int_{x_0}^s \min\left(0,\y_\infty(z)(z+\y_\infty (z)^2)\right) dz ds.
	\end{align*}
	and the latter takes a  finite value.
\end{proof}

\subsection{Proof of Theorems \ref{thm:ybounds} and \ref{thm:ineq}}

\begin{proof}
	Note that Theorem \ref{thm:ybounds} is equivalent to Theorem \ref{thm:ubounds}. Note that the monotonicity  follows  directly from Theorem \ref{thm:monotonicity}. The fact that $ b \mapsto \omega(b)$ is a bijection follows from Lemmas \ref{lem:bto0}
 and \ref{lem:btoinfty}, concluding the proof of Theorem \ref{thm:ubounds}.

	Note that Theorem \ref{thm:ineq} is equivalent to Theorem \ref{thm:inequ}. The latter follows by Theorem \ref{thm:monotonicity}, the positivity of $\y_0$ and Theorem \ref{thm:lowerBnd1} \ref{thm:ubounds}.
 	\end{proof}


\section{Proof of Theorem \ref{thm:last}}

In this section we prove Theorem \ref{thm:last} which states that there is a unique solution to the d-PII \eqref{eq:dPII} with boundary condition \eqref{eq:BC}.

 We start with the following lemma.

\begin{lemma} \label{lem:posdef}
	For $1\leq n$ consider symmetric the tridiagonal matrix $D=D(s_1,\ldots,s_{n-1})$ defined by 
	$$
	D_{jj}=\frac{\sqrt{s_{j-1}}+\sqrt{s_{j+1}}}{s_j^{3/2}}, \qquad j=1,\ldots, n-1,
	$$
	and 
	$$
	D_{j(j+1)}=-\frac{1}{\sqrt{s_j s_{j+1}}}, \qquad j=1,\ldots, n-1,
	$$
	where we have set $s_0=s_n=1$. Then $D$ is a positive definite  matrix.
\end{lemma}
\begin{proof}
	The matrix can be written as a sum of non-negative definite matrices that consists of $2\times2$. Indeed, set
	$$
	H_j=\begin{pmatrix}
		\frac{	\sqrt{s_{j+1}}}{s_j^{3/2}} & -\frac{1}{\sqrt{s_j s_{j+1}}} \\
		- \frac{1}{\sqrt{s_j s_{j+1}} }& \frac{	\sqrt{s_{j}}}{s_{j+1}^{3/2}}
	\end{pmatrix}.
	$$
	Then it is easily verified that $H_j$ is positive semi-definite. By writing $D$ as 
	$$
	D= \begin{pmatrix}
		\frac{1}{s_1^{3/2}} & 0\\ 
		0& 0
	\end{pmatrix}+\begin{pmatrix} H_1 & 0\\
		\\ 
		0& 0
	\end{pmatrix}+ \begin{pmatrix} 0 & 0& 0\\
		0& H_2 & 0\\ 
		0& 0 &0
	\end{pmatrix}+ \cdots +\begin{pmatrix} 0 & 0\\
		\\ 
		0& H_{n-2}
	\end{pmatrix}+\begin{pmatrix}
		0 & 0\\ 
		0& \frac{1}{s_{n-1}^{3/2}}
	\end{pmatrix}
	$$
	we see that $D$ is a sum of positive semi-definite matrices. This immediately implies that $D$ is positive semi-definite. Furthermore, from the sum representation it is also not hard to see that $x^TDx=0$ iff $x=0$. 
\end{proof}
We recall the definition of the Hamiltonian $H_n$ in \eqref{eq:aHamiltonian}. 
\begin{proposition} \label{prop:last}
	The Hamiltonian $H_n$ has a unique minimizer on the space $a_k\in (0,1)$ for $k=0,\ldots, n-2$ with boundary conditions $a_{-1}=a_{n-1}=1$, and that minimizer occurs at its only critical point.
\end{proposition}

\begin{proof}
	Recall that 
	\begin{equation*}
		H_n(a)= -\sum_{k=0}^{n-2} \frac{k+\shift}{t}\log(1-a_k^2)-\sum_{k=0}^{n-1} a_ka_{k-1}
	\end{equation*}
	with boundary conditions $a_{-1} = a_{n-1} = 1$. After the change of variable $a_k=\sqrt{s_{k+1}}$ we see that 
	$$
	\tilde H_n(s)= -\sum_{k=1}^{n-1}\frac{k+\shift}{t}\log(1-s_k)-\sum_{k=1}^{n}\sqrt{ s_ks_{k-1}},
	$$
	for $s_k\in (0,1)$ for   $k=1,\ldots, n-1$ with boundary conditions $s_{0}=s_{n}=1$.  Any local minimizer for $H_n$  in the space $a_k\in [0,1]$ for $k=0,\ldots, n-2$ with boundary conditions $a_{-1}=a_{n-1}=1$, gives a local minimizer for $\tilde H_n$ on the space $s_k\in [0,1]$ for   $k=1,\ldots, n-1$ with boundary conditions $s_{0}=s_{n}=1$, and vice versa.

	By compactness and continuity, $\tilde H_n$ has a minimizer. Let $s^*$ be a minimizer. We will show that it is unique and lies in the interior $(0,1)^n$. For the latter, note that $\tilde H_n(s_1,\ldots,s_n)=+\infty$ if $s_k=1$ for some $k=1,\ldots,n$, and thus $s^*_k <1$ for $k=1,\ldots,n$. 
	We next show that also $s^*_k>0$ for $k=1,\ldots,n$. Suppose the opposite and let $k^*$ be the first integer such that $s^*_{k^*}=0$. Then $\partial \tilde H_n/ \partial s_{k^*}(s^*+\varepsilon e_{k^*})\sim -\sqrt{(s^*_{k-1}+s^*_{k^*+1})/\varepsilon}$ and thus  $ \tilde H_n(s^*+\varepsilon e_{k^*})<\tilde H_n(s^*)$ for sufficiently small $\varepsilon$, giving a contradiction.  We conclude that the minimizer $s^*$ for $\tilde H_n$ must be in the interior ${s^*_k}\in (0,1)$ for all $k=1,\ldots,n$.  It remains to show it is unique.  To this end, observe that the Hessian of $H_n$ is given by 
	$$
	\begin{pmatrix}
		\frac{\shift}{t (1-s_1)^2}&\\
		& 	\frac{1+\shift}{ t(1-s_2)^2}\\
		&& \ddots \\
		&&&	\frac{n-2+\shift}{t(1-s_{n-1})^2}
	\end{pmatrix}+\tfrac 14 D
	$$
	which is a sum of diagonal matrix with positive entries and $D$ is the matrix from Lemma \ref{lem:posdef}. It is thus positive-definite and therefore $\tilde H_n$ is a strictly convex function. Therefore the minimizer is unique.
	
	This also implies that $H_n$ has a unique minimizer in the interior of the domain. Since any critical point of $H_n$ corresponds to a critical point of $\tilde H_n$, we see that $H_n$ has a unique critical point minimizing $H_n$.
\end{proof}
We are now ready to prove Theorem \ref{thm:last}.
\begin{proof}[Proof of Theorem \ref{thm:last}]
This follows directly from Proposition \ref{prop:last} and the fact that $\nabla H_n=0$ is the same as the d-PII equation. 
\end{proof}

\section{Proof of Theorem \ref{thm:main}}
\label{sec:proofOfMain}

The main idea behind the proof of Theorem \ref{thm:main} is to first construct an approximate minimizer $\azero$ of the Hamiltonian $H_n$.  Indeed,  we will construct $\azero$ so that $\nabla H(\azero) = \mathcal{O}(n^{-4/3})$.  To prove that $a^{(0)}$ is close to the minimizer, we use the Kantorovich theorem for Newton's method \cite{Kant}, which gives sufficient conditions for ensuring that for a function $F(x)$ and a point $x_0$ there exists a solution to $F(x) = 0$ in a neighborhood of $x_0$. This will allow us to show that the true minimizer $a^*$ lies in the ball $B(a^{(0)},b n^{-2/3})$.  We then  conclude by observing that under the appropriate rescaling we have $n^{1/3}a^* \in B(n^{1/3}a^{(0)},b n^{-1/3})$, and so taking the limit as $n\to \infty$ we will have
\[
\lim_{n\to \infty, k/n^{1/3} \to x} n^{1/3}a^*_{n+k} = \lim_{n\to \infty} n^{1/3}\azero_{n+k}.
\]
From our explicit form of $a^{(0)}$ this limit can easily be computed and this will prove Theorem \ref{thm:main}.


\subsection{The approximate minimizer $\azero$ and proof of Theorem \ref{thm:main}}

It will be convenient for the rest of this paper to work with a shifted version of $\y_b$. As the position of the pole $\omega(b)$ is strictly monotone in $b$ we can invert this and consider $b$ to be a function of $\omega$. Define
\begin{equation} \label{eq:defnuu}
\Y_\omega(x) = \y_{b(\omega)}(x+\omega)
\end{equation}
where $\y$ satisfied \eqref{eq:PIIalt} with the given boundary conditions. The new $\Y_\omega$ has a pole at $x = 0$ for all choices of $\omega,$ and we indicate the $\omega$ dependence in the notation to emphasize this dependence. We will use this notation for a generic $\omega$ not just for the particular one introduced as a parameter in the model. Notice that $\Y_\omega$ solves the ODE
\begin{equation}
\label{eq:shiftedPII}
\Y''_{\omega}(x) = 2(x+\omega)\Y_\omega(x) + 2\Y_\omega^3(x).
\end{equation}

To define $\azero$ we begin by introducing the parameter
\begin{equation}
\label{eq:omegan}
\omega_n = \omega + \frac{\shift}{n^{1/3}}.
\end{equation}
The approximate minimizer of $H_n(a)$ is then constructed from a discretization of $\Y_{\omega_n}$ in the following way:
\begin{equation}
\label{eq:azero}
\azero_k = \frac{1}{n^{1/3}} \Y_{\omega_n} \left(\frac{k-n}{n^{1/3}}\right) - \frac{\omega_n}{3(k-n) n^{2/3}} - \frac{1}{2n}
\end{equation}
for $k = 1,\ldots, n-2$.  We  will show that $a^{(0)}$ is close to the unique critical point $a^*$ of $H_n$. But instead of working directly with  the gradient function $\nabla H_n(a)$, we will take a modification and consider $F$ defined by
\begin{equation}
	\label{eq:F}
	\left(F(a)\right)_k = -(1-a_k^2) \left(\nabla H_n(a)\right)_k.
\end{equation}
This modification still identifies the minimizer because for $a_k \in (0,1)$ we will have that 
\[
F(a) = 0 \quad \iff \quad \nabla H_n(a) = 0.
\]
The following theorem is the key to the proof of Theorem \ref{thm:main}.

\begin{theorem}
\label{thm:Fzero}
With $F$ as in \eqref{eq:F} and $a^{(0)}$ as in\eqref{eq:azero}, there exists a $c>0$ such that for each $n$ sufficiently large there exists an $a^* \in B(a^{(0)},c n^{-2/3})$ such that $F(a^*)=0$. 
\end{theorem}
We postpone the proof of this theorem to the next subsection.

Note that it does not immediately follow that the point  $a^*$ from Theorem \ref{thm:Fzero}  is also the unique critical point (and  unique minimizer) of $H_n$ in the domain $a_k\in (0,1)$ for $k=0,\ldots, n-2$. Indeed, since the ball $B(a^{(0)},c n^{-2/3})$ lies partly outside our domain and it could be possible that $a^*_k \notin (0,1)$ for some $k$. In the following proposition we show that this does not happen and that $a^*$ is indeed the unique minimizer of $H_n$.

\begin{proposition}
	Let $a^*$ be the solution to $F(a) = 0$ that lies in the ball $B(a^{(0)},c n^{-2/3})$. Then for $n$ sufficiently large we have $0<a^*_k < 1$ for all $k = 0,1,...,n-2$.
\end{proposition}

\begin{proof}
	We start by showing that $0<a^*_k$ for $n$ sufficiently large. Recall that $\Y_\omega(x)>0$ on $(-\infty,0)$, a statement that can be strengthened to $y_\omega (x)>\delta$ for some $\delta>0$ due to the boundary conditions of $\Y_\omega$. We then note that 
	\[
	\azero_k = \frac{1}{n^{1/3}}\Y_{\omega_n} \left(\frac{k-n}{n^{1/3}}\right) + \mathcal{O}(n^{-2/3}) > \frac{\delta}{n^{1/3}} + \mathcal{O}(n^{-2/3}).
	\]
	Then since $a^*\in B(a^{(0)},c n^{-2/3})$ we can choose $n$ sufficiently large so that $a^*_k >0$.
	
	Now to show $a^*_k<1$, begin by observing that $a^*_k$ is close to $1$ only for $k < n-tn^{1/3}$ for some $t$. Once we reach the region where $n-k = \mathcal{O}(n^{1/3})$ the approximate solution $\azero$ follows the behavior of the rescaled Painlev\'e solution and descends towards $0$. On this region it follows that $a^*\in B(a^{(0)},c n^{-2/3})$ ensures that $a^*_k< 1$.
	For $k$ small, i.e. the region where $\azero$ is close to $1$, we start by observing that
	\[
	\left(F(a)\right)_k = -(1-a_k^2) \left(\nabla H_n(a)\right)_k = (1-a_k^2)(a_{k+1} + a_{k-1}) - 2 \frac{k+\shift}{t} a_k.
	\]
	Since $a^*$ solves $F(a) = 0$ it follows that 
	\[
	\big(1-(a^*_k)^2\big)(a^*_{k+1} + a^*_{k-1}) - 2 \frac{k+\shift}{t} a^*_k = 0.
	\]
	Rewritten we have
	\[
	\big(1-(a^*_k)^2\big)(a^*_{k+1} + a^*_{k-1}) = 2 \frac{k+\shift}{t} a^*_k
	\]
	where the right hand side must be strictly positive since $a^*$ is close to $\azero$ which is close to $1$ and  $\shift\ge 0$. This means there is no solution where $a^*_k = 1$ or the left side would be $0$ and equality couldn't hold. Similarly we can never have $a^*_k>1$, otherwise we would need 
	\[
	a^*_{k+1} + a^*_{k-1}<0,
	\]
	meaning that at least one of $a^*_{k+1}$ and $a^*_{k-1}$ is negative. However this is impossible since $a^* \in B(a^{(0)},c n^{-2/3})$ and $\azero$ is close to $1$ in this region.
\end{proof}

We are now ready to prove the main result of this paper. 

\begin{proof}[Proof of Theorem \ref{thm:main}]
	Let $a^*$ be the unique minimizer of $H_n$ over the set $a_k\in (0,1)$ for $k=0,\ldots, n-2$ with $a_{-1}= a_{n-1}=0$ from Proposition \ref{prop:last}. Then with $a^{(0)}$ as \eqref{eq:azero} we have
	\[
	\lim_{n\to \infty, k/n^{1/3}\to x} n^{1/3} a^*_{n+k}  = \lim_{n\to \infty, k/n^{1/3}\to x} \left(n^{1/3} (a^*_{n+k} - \azero_{n+k}) +n^{1/3} \azero_{n+k}\right)
	\]
	By Theorem \ref{thm:Fzero} we have $a^*\in B(a^{(0)},c n^{-2/3})$ it follows that $n^{1/3} (a^*_{n+k} - \azero_{n+k}) = \mathcal{O}(n^{-1/3})$, so that term vanishes in the limit. That leaves us with 
	\[
	\lim_{n\to \infty, k/n^{1/3}\to x} n^{1/3}\azero_{n+k} = \lim_{n\to \infty, k/n^{1/3}\to x}\left( \Y_{\omega_n}\left(\frac{k}{n^{1/3}}\right) + \mathcal{O}(n^{-1/3})\right)
	\]
	Observe that $\omega_n \to \omega$ monotonically, which also give $b(\omega_n)\to b(\omega)$ monotonically. Since solutions of the PII are meromorphic in $b$ and $x$ we find:
	\begin{align*}
		\lim_{n\to \infty, k/n^{1/3}\to x}& \left(\Y_{\omega_n}\left(\frac{k}{n^{1/3}}\right)- \Y_{\omega_n}(x)\right) + \left(\y_{b(\omega_n)}(x+\omega_n) -\y_{b(\omega_n)}(x+\omega)\right)\\
		& \hspace{1cm} +\left(\y_{b(\omega_n)}(x+\omega)-\y_{b(\omega)}(x+\omega)\right) + \y_{b(\omega)}(x+\omega)\\
		&= \y_{b(\omega)}(x+\omega). 
	\end{align*}
	Here we are using the relationship $\Y_\omega(x) = \y_{b(\omega)}(x+\omega)$ for all $\omega \in (-\infty,\infty)$.
	This concludes the proof of Theorem \ref{thm:main}. 
\end{proof}
\subsection{Proof of Theorem \ref{thm:Fzero}}

The proof of Theorem \ref{thm:Fzero} will be a direct application of the  Kantorovich Theorem \cite{Kant} for Newtons methods in Banach spaces. There are many variations to this theorem and we refer to \cite{Yam} for a detailed discussion and  historical overview.  The version that we state here is a slight simplification of     Theorem 2.2 combined with Remark 2.3 in \cite{Yam}. 
\begin{theorem}
\label{thm:kantorovich}
Let $\mathcal B$ be a Banach space and $F:X \to \mathcal B$ a differentiable function on an open convex subset $X \subset \mathcal B$.  Let $x_0 \in X$ be such that $F(x_0)\neq 0$ and $(F'(x_0))^{-1}$ exists. Furthermore,  assume that there exists $\beta, M$ such that 

\begin{enumerate}[(1)]
	\item $\|\left(F'(x_0)\right)^{-1} F(x_0)\| \leq  \beta$, \label{item:1}
	\item $\|\left(F'(x_0)\right)^{-1} \left(F'(x)-F'(y)\right)\| \leq M\|x-y\|$ for all $x, y \in B(x_0,2\beta)\subset X$, \label{item:2}
	\item $\beta M < \frac{1}{2}$ \label{item:3}
\end{enumerate}
Then, as $k \to \infty$, the Newton iteration $x_{k+1}=x_k-F(x_k)^{-1} F(x_k)$ starting from $x_0$,  converges to an $x^* \in B(x_0,t^*)$ such that $F(x^*)=0$.  Here $t^*=2 \beta/(1+\sqrt{1-2\beta M})<2 \beta$.
\end{theorem}

We will also make use of the following bounds.

\begin{theorem}
\label{thm:Fbounds}
With $F$ as in \eqref{eq:F} and $a^{(0)}$ as in\eqref{eq:azero}, there exists a $c>0$ such that  
\begin{enumerate}[(i)]
	\item $F$ satisfies $$ \left\|\left(F'(a^{(0)})\right)^{-1}F(a^{(0)})\right\|_\infty \leq \frac c2 n^{-2/3}$$ \label{item:i}
	\item for $a,b \in B(a^{(0)}, cn^{-2/3})$, 
	$$
	\left\|\left(F'(a^{(0)})\right)^{-1}\left(F'(a))-F'(b)\right)\right\|_\infty\le M \|a-b\|_\infty
	$$
	where $M=\mathcal O(n^{1/3})$ as $n \to \infty$. \label{item:iii}
\end{enumerate}
\end{theorem}
This theorem follows from a long technical argument and is proved in Section \ref{sec:Fboundsproof}.

\begin{proof}[Proof of Theorem \ref{thm:Fzero}]
It follows directly from Statement \textit{\ref{item:i}} in the previous theorem that $F$ satisfies condition \textit{\ref{item:1}} in the Kantorovich Theorem (\ref{thm:kantorovich}) in a neighborhood of $\azero$.

From \textit{\ref{item:iii}} we can see that $F$ satisfies \textit{\ref{item:2}} with $M = \mathcal{O}(n^{1/3})$. Finally to conclude we note that 
\[
\beta M = \mathcal{O}(n^{-1/3})<\frac{1}{2},
\]
if $n$ is sufficiently large. 

Since all the conditions of the Kantorovich theorem hold we have the desired conclusion for all big enough $n$ (i.e. when $\beta M < 1/2$).\end{proof}



\section{The approximate minimizer}
\label{sec:regionI}
It remains to prove Theorem \ref{thm:Fbounds}. In this section we will prepare its proof by showing that $\|F(a)\|_\infty$ is small. The  rest of the proof of Theorem \ref{thm:Fbounds} will then be given  in Section \ref{sec:Fboundsproof}.

Recall from  \eqref{eq:azero} that we defined
\[
\azero_k = \frac{1}{n^{1/3}} \Y_{\omega_n} \left(\frac{k-n}{n^{1/3}}\right) - \frac{\omega_n}{3(k-n) n^{2/3}} - \frac{1}{2n},
\]
with $\Y_{\omega_n}$ as in \eqref{eq:defnuu}.  This is a global solution, but the correction terms (the terms other than $\Y_{\omega_n}$) only begin to play a role as $k$ gets close to $n$, and were in fact introduced to improve the error in that region. This region where $k$ is close to $n$ will need to be handled separately. To this end we introduce two regions which will be studied separately:
\begin{itemize}
	\item Region I: $k = 0, \ldots, n- \lfloor \Rdelta n^{1/3}\rfloor$
	\item Region II: $k = n- \lfloor \Rdelta n^{1/3}\rfloor, \ldots, n-2$.
\end{itemize}

\begin{proposition}
\label{prop:Fbounds1}
There exists $c>0,$ uniform in $k$ and $n$ such that 
\[
\left( F(a^{(0)})\right)_k \le cn^{-4/3}.
\]
\end{proposition}
This is proved in Subsection \ref{subsec:gradboundI} for $k$ in Region~I, and in Subsection \ref{subsec:regionII} for $k$ in Region~II.


\subsection{Bounding $F(\azero)$ on Region I}
\label{subsec:gradboundI}

For simplification we will use the following notations:
\begin{equation}
\Y(x) =\Y_{\omega_n}(x)  \qquad \text{ and } \qquad x_k = \frac{k-n}{n^{1/3}}.
\end{equation}
We will need the following lemma.
\begin{lemma}
	\label{lemma:secondDiff}
	For $k$ in Region I we have that
	\begin{enumerate}[(1)]
		\item \[
		\azero_{k+1} - 2\azero_k + \azero_{k-1}  = \frac{\Y''(x_k)}{n} + \mathcal{O}(n^{-5/3})
		\]
		
		\item \[
		\left(\azero_k\right)^2 \left(\azero_{k+1} + \azero_{k-1}\right) = 2\frac{\Y^3(x_k)}{n} + \mathcal{O}(n^{-5/3})
		\]
		
		\item \[
		2\azero_k\left(\frac{n-k -\omega_n n^{1/3} }{N}\right) = -2(x_k+\omega_n)\frac{\Y(x_k)}{n} + \mathcal{O}(n^{-4/3})
		\]
	\end{enumerate}
	as $n \to \infty$,  where the constants in the $\mathcal O$ terms is uniform in $k$.
\end{lemma}

\begin{proof}[Proof of Lemma \ref{lemma:secondDiff} \textit{(1)}]
	We begin by using a Taylor estimate:
	\[
	\Y(x_{k\pm1}) = \Y(x_k) \pm \frac{\Y'(x_k)}{n^{1/3}} + \frac{\Y''(x_k)}{2n^{2/3}} \pm \frac{\Y'''(x_k)}{6n} + \frac{\Y^{(4)}(\xi)}{4! n^{4/3}}.
	\]
	This give us that:
	\[
	\frac{1}{n^{1/3}} \left(\Y\left(x_{k+1}\right)-2\Y\left(x_{k}\right)+ \Y\left(x_{k-1}\right)\right) = \frac{\Y''(x_k)}{n} + \mathcal{O}(n^{-5/3}).
	\]
	Here the error term may be bounded in absolute value for all $k$ in Region I by 
	\[
	n^{-5/3} \sup_{x\in (-\infty,-\Rdelta]}{|\Y^{(4)}(x)|}{},
	\]
	and from the asymptotics for $\nu(x)$  for $x \to -\infty$  the supremum must be finite. 
	This gives us the contribution from the main terms of $\azero_k$. For the correction terms we note that the final $-\frac{1}{2n}$ terms cancel perfectly, leaving us with:
	\begin{align*}
		\Bigg|\frac{\omega_n}{3n^{2/3}} &\left( \frac{1}{k+1-n} -   \frac{2}{k-n}+  \frac{1}{k-1-n}\right)\Bigg| \\
		& \hspace{1.5cm} =\Bigg| \frac{2\omega_n}{3n^{2/3}(k+1-n)(k-n)(k-1-n)}\Bigg| \le \frac{2|\omega_n|}{3n^{5/3}\Rdelta^3}
	\end{align*}
	for $k$ in Region I. 
\end{proof}

\begin{proof}[Proof of Lemma \ref{lemma:secondDiff} \textit{(2)}]
	The proof of this is similar to the proof of  \textit{(1)}. We Taylor expand the main terms to find that 
	\[
	\Y(x_{k+1})+\Y(x_{k-1}) = 2\Y(x_k) + \mathcal{O}(k^{-2/3}).
	\]
	Since the leading term of $\azero$ is on the order of $n^{-1/3}$ and the correction terms are on the order of $1/n$ it follows that any cross term of $(\azero_k)^2(\azero_{k+1}+ \azero_{k-1})$ other than the main terms will be $\mathcal{O}(n^{-5/3})$. Careful tracking of the form of these terms shows us that we can bound the magnitude of the error terms uniformly in $k$ for all $k$ in Region I.
\end{proof}

\begin{proof}[Proof of Lemma \ref{lemma:secondDiff} \textit{(3)}]
	We begin by recognizing that 
	\[
	\frac{n-k -\omega_n n^{1/3} }{N} = -\frac{x_k+\omega_n}{n^{2/3}}\left(1 + \frac{\omega_n n^{1/3}}{t}\right)
	\]
	Which gives us
	\[
	2\azero_k\left(\frac{n-k -\omega_n n^{1/3} }{t}\right) = -2(x_k+\omega_n)\frac{\Y(x_k)}{n} - \frac{\omega_n(x_k+\omega_n)}{3x_k n^{5/3}}  -2\frac{x_k+\omega_n}{n^{5/3}}+ \mathcal{O}(n^{-5/3}).
	\]
	We can check that the second term on the right-hand-side is also on the order of $\mathcal{O}(n^{-5/3}),$ however the third term will be $\mathcal{O}(n^{-4/3})$ for $k$ small (since $x_k = (n-k)/n^{1/3}$). The magnitude of this term can be bounded uniformly in $k$ by $Cn^{-4/3}$ for some constant $C$.
\end{proof}
Now we are for the proof of Proposition \ref{prop:Fbounds1} on Region I.

\begin{proof}[Proof of Proposition \ref{prop:Fbounds1} on Region I]
Recall that 
\[
 \left(F(a)\right)_k  = (1-a_k^2)(a_{k+1} + a_{k-1}) - 2 \frac{k+\shift}{t} a_k.
\]
We begin by observing that $ \left(F(a)\right)_k$ can be rewritten as
\[
 \left(F(a)\right)_k =  a_{k+1} - 2a_k + a_{k-1}  - a_k^2 (a_{k+1} + a_{k-1}) + 2a_k\left(\frac{n-k -\omega_n n^{1/3} }{t}\right),
\]
where $t = n - \omega n^{1/3}$ and $\omega_n = \omega + \frac{\shift}{n^{1/3}}$ as before.

Lemma \ref{lemma:secondDiff} gives us that 
\[
 \left(F(\azero)\right)_k = \frac{1}{n}\left(\Y''(x_k) - 2\Y(x_k)(x_k + \omega_n) - 2\Y^3(x_k)\right) + \mathcal{O}(n^{-4/3})
\]
Now since $\Y$ satisfies \eqref{eq:shiftedPII} we find that the main terms cancel perfectly giving us the desired proposition.
\end{proof}


\subsection{Bounding $F(\azero)$ on Region II}
\label{subsec:regionII}

On Region II, not only do the correction terms become relevant, but we also have that $\Y_\omega(x)$ is close to a pole. Indeed the behavior of $\Y_\omega(x)$ near $0$ is given by the following lemma:
\begin{lemma}
\label{lemma:yexpansion}
The function $\Y_\omega$ has an expansion 
\[
\Y_\omega(x)  = -\frac{1}{x} \left(1 - \frac{\omega}{3}x^2 - \frac{x^3}{2} + \sum_{k=4}^\infty c_k x^k\right).
\]
\end{lemma}
\begin{proof}
Since $\Y_\omega$ is meromorphic with a simple pole at zero with residue $-1$ it follows that it has an expansion of the form 
\[
\Y_\omega(x)  = -\frac{1}{x} \left(1  + \sum_{k=1}^\infty c_k x^k\right).
\]
We can then determine the first coefficients by inserting this expansion into the equation above.
\end{proof}

In order to treat this region we will need to do consider transformed $\Y_\omega(x)$ and $\azero_k$ that remove this singularity. 

\subsubsection{Transforming the problem}

We begin by removing the singularity from $\Y_\omega(x)$:
\begin{lemma}
Let $\y$ be the solution to \eqref{eq:PIIalt} with the given boundary conditions. In particular recall that $\y$ has a simple pole at $\omega$, and set 
\begin{equation}
\label{eq:f}
f_{\omega}(x) = x\y(x+\omega) +1 = x\Y_\omega(x)+1.
\end{equation}
Then $f_{\omega}$ satisfies the following ODE:
\begin{equation}
\label{eq:ODEf}
x^2 f_\omega''(x) - 2x f_\omega'(x) = 2\big(f_\omega(x)-1\big)\Big(f_\omega^2(x) - 2f_\omega(x) + x^2(x+\omega)\Big)
\end{equation}
\end{lemma}

\begin{proof}
The verification can be done by taking derivatives of $f_\omega$ and using \eqref{eq:PIIalt} to make the appropriate substitutions.
\end{proof}

Notice that the $f_\omega$ in the previous lemma satisfies also that 
\[
\frac{f_\omega(x) -1}{x} = \y(x+\omega) = \Y_\omega(x).
\]
This leads us to the natural introduction of $b_k$ defined by
\begin{equation}
\label{eq:bk}
\azero_{n+k} = \frac{b_k -1}{k}, \qquad \text{ or } \qquad b_k = k\azero_{n+k} +1
\end{equation}
where $k$ plays the role of $x$. We make this connection more concrete by defining $x_k = \frac{k}{n^{1/3}}$ which leads us to 
\[
b_k = f_{\omega_n} (x_k) - \frac{\omega_n}{3n^{2/3}} - \frac{k}{2n}.
\]
\begin{remark}
With this reindexing, Region II corresponds to $- \delta n^{1/3} \le k \le -2$. It's worth keeping in mind that $k$ will always be negative.
\end{remark}

\begin{lemma}
\label{lemma:fexpansion}
The function $f_\omega$ has the following expansion at $0$:
\begin{equation*}
f_\omega(x) = \frac{\omega}{3}x^2 + \frac{x^3}{2} + \mathcal{O}(x^{4}), 
\end{equation*}
and therefore for $k =  -\lfloor \delta n^{1/3}\rfloor,\ldots, -2,$
\[
|b_k| \le C_\delta \frac{k^2}{n^{2/3}}.
\]
\end{lemma}

\begin{proof}
The expansion of $f$ follow from the definition together with Lemma \ref{lemma:yexpansion}. The bound on $b_k$ follows because we need only consider the expansion of $f_\omega$ in a compact neighborhood of $0$.
\end{proof}

\subsubsection{Bounding $(F(a))_{n+k}$}

We begin by rewriting our function $F(a)$ and it's derivative in terms of our $b_k$'s. 

\begin{lemma} Using $x_k = \frac{k}{n^{1/3}}$ we can rewrite $\left(F(\azero)\right)_{n+k}$ in terms of the $b_k$'s as:
\begin{align}
\left(F(\azero)\right)_{n+k} &= \frac{(b_{k+1} - 2b_k+ b_{k-1})}{k} - \frac{b_{k-1}-b_{k+1}}{k^2} \notag \\
&\quad - \left[\frac{b_{k+1}+ b_{k-1}}{k}+\frac{b_{k-1}-b_{k+1}}{k^2}-\frac{2}{k}  \right]\frac{b_k^2-2b_k}{k^2-1} \notag\\
&\quad - 2 \frac{b_k-1}{k}\cdot \frac{x_k+\omega_n}{n^{2/3}}  + \mathcal{O}(n^{-4/3}), \label{eq:F(b)}
\end{align}
as $n \to \infty$,  where the constant in the $\mathcal O(n^{-4/3})$ term is uniform in $k = n- \lfloor \delta n^{1/3}\rfloor, \ldots, n-2$   (but can depend on $\delta$ and $\omega$).
\end{lemma}

\begin{proof}
From the definition of $F(\azero)$ we get
\begin{align}
\left( F(\azero) \right)_{n+k} &= \left(\azero_{n+k+1} + \azero_{n+k-1}\right)(1-(\azero_{n+k})^2) - 2\azero_{n+k} \left(\frac{n+k + \shift}{t} \right)\notag \\
& = \left(\azero_{n+k+1} + \azero_{n+k-1}\right)(1-(\azero_{n+k})^2) - 2\azero_{n+k}\notag \\
& \qquad \qquad  - 2\azero_{n+k}\frac{k+\omega_n n^{1/3}}{n} + \mathcal{O}(n^{-4/3}). \label{eq:F(\azero)}
\end{align}
The final error term comes from
\[
\frac{n+k+\shift}{t} - 1 = \frac{k+ \omega_n n^{1/3}}{n}\left(1 - \frac{\omega n^{1/3}}{n+\omega n^{1/3}}\right).
\]

Now, to rewrite $F(\azero)$ in terms of the $b_k$'s we begin by rewriting individual pieces:
\begin{align}
\azero_{n+k+1} +\azero_{n+k-1} & = \frac{k(b_{k+1}+b_{k-1}) + b_{k-1}-b_{k+1} - 2k}{k^2-1}  \label{eq:beq1}\\
1-(\azero_{n+k})^2 & = \frac{k^2-1}{k^2} - \frac{b_k^2-2b_k}{k^2}  \label{eq:beq2}
\end{align}
Together these give us that 
\begin{align*}
\big(\azero_{n+k+1}& + \azero_{n+k-1}\big)(1-(\azero_{n+k})^2) - 2\azero_{n+k} \\
&= \frac{(b_{k+1} - 2b_k+ b_{k-1})}{k} - \frac{b_{k-1}-b_{k+1}}{k^2}\\
 &\qquad \quad - \frac{b_k^2-2b_k}{k^2}\left[ \frac{k(b_{k+1}+b_{k-1}) + b_{k-1}-b_{k+1} - 2k}{k^2-1}\right]
\end{align*}
where the final line comes from the product of second term in \eqref{eq:beq2} with the terms in \eqref{eq:beq1}. To complete the proof we just need to rewrite the final $\azero_{n+k}$ in \eqref{eq:F(\azero)} in terms of $b_k$ and substitute in $x_k = \frac{k}{n^{1/3}}$.
\end{proof}

We now need to show that the right hand side of \eqref{eq:F(b)} is a discretization of the differential equation for $f_{\omega_n}$ \eqref{eq:ODEf}.

	\begin{lemma} \label{lem:frombktof2}
		We have
		$$
			\frac{k(b_{k+1}+b_{k-1}-2b_k) -(b_{k+1}-b_{k-1})}{k^2}
			= 
			\frac{x_k^2 f''(x_k)-2 x_k f'(x_k)}{k^3} + \mathcal O(n^{-4/3}),
		$$
		and 
		\[
		b_{k-1}-b_{k+1} = -\frac{2}{n^{1/3}} f'(x_k) + \frac{1}{n} + \mathcal O(n^{-4/3}).
		\]
as $n \to \infty$,  where the constant in the $\mathcal O(n^{-4/3})$ term is uniform in $k = n- \lfloor \delta n^{1/3}\rfloor, \ldots, n-2$ (but depends on  $\delta$).
	\end{lemma}
	\begin{proof}
From a simple Taylor expansion we get 
	$$	f(x_{k\pm1})=f(x_k)\pm \frac{f'(x_k)}{n^{1/3}}+\frac{f''(x_k)}{2n^{2/3}}\pm\frac{f'''(x_k)}{6n}+ \mathcal O(n^{-4/3}),$$
	Now note that from Lemma \ref{lemma:yexpansion} we find  that $f'''(x)=3 + \mathcal O(x)$. Therefore,
\begin{multline*}
		k(f(x_{k+1})+f(x_{k-1})-2f(x_k)) -(f(x_{k+1})-f(x_{k-1}))
			\\ = \frac{x_k^2 f''(x_k)-2x_k f'(x_k)}{k}+ \frac{1}{n}+ \mathcal O(\delta n^{-4/3}).
\end{multline*}
And here we see another reason that the choice $b_k = f(x_k)$ isn't good enough. Indeed, we need to cancel the term $\frac{1}{n}$ if we want to obtain an error term of order $\mathcal O(n^{-4/3})$. This cancellation comes from the fact that while the extra terms cancel in the second order difference  for the first order difference we have:
\[
b_{k-1}-b_{k+1} = f(x_{k-1})-f(x_{k+1})+ \frac{1}{n}.
\]
The second equality is proved in using exactly the same tools.
	\end{proof}

\begin{lemma} \label{lem:adjustingbk}
	For $m\geq 1$, we have 
		$$	\frac{k^2}{k^2-1} b_k^m=f(x_k)^m+\mathcal O\left(\frac{k^2}{n^{4/3}}\right)$$
		 where the error term can be bounded uniformly for $-2\ge k\ge -\delta n^{1/3}$ as $n \to \infty$.
	\end{lemma}
	\begin{proof}
		First we deal with the case $m=1$. Then 
		$$ 
		b_k=f(x_k)-\frac{\omega_n}{3 n^{2/3}}-\frac{k}{2n}=f(x_k)-\frac{x^2_k \omega_n}{3k^2}-\frac{x^3_k}{2k^2}=\frac{k^2-1}{k^2} f(x_k)+	\mathcal O\left(\frac{k^{2}}{n^{4/3}}\right),$$
		from which the statmement follows. 
		For $m>1$ we fist note that 
		$$f(x_k)^m=\mathcal O\left(\frac{k^{2m}}{n^{2m/3}}\right),$$
		which implies that  
		$$\left(\frac{k^2-1}{k^2}\right)^{m-1}f(x_k)^m=f(x_k)^m+\mathcal O \left(\frac{k^{2m-2}}{n^{2m/3}}\right)=f(x_k)^m+\mathcal O\left(\frac{k^{2}}{n^{4/3}}\right)$$
	and this proves the statement for $m\geq 2$.
	\end{proof}

\begin{lemma} \label{lem:frombktof3}
	We  have
\begin{align*}
\bigg[\frac{b_{k+1}+ b_{k-1}}{k}+&\frac{b_{k-1}-b_{k+1}}{k^2}-\frac{2}{k}\bigg] \frac{b_k^2-2b_k}{k^2-1}\\
&  = 2\frac{f(x_k)-1}{k^3}\Big(f^2(x_k)-2f(x_k)\Big) + \mathcal{O}(n^{-4/3}),
\end{align*}
as $n \to \infty$,  where the constants in the $\mathcal O(n^{-4/3})$ term is uniform in $k = n- \lfloor \delta n^{1/3}\rfloor, \ldots, n-2$.
\end{lemma}

\begin{proof}
We begin by rewriting the quantity in brackets in terms of $f$:
\begin{align*}
\frac{b_{k+1}+ b_{k-1}}{k}&+\frac{b_{k-1}-b_{k+1}}{k^2}-\frac{2}{k}\\
&  = 2\frac{f(x_k)-1}{k}+ \frac{f''(x_k)}{k n^{2/3}} - \frac{2f'(x_k)}{k^2 n^{1/3}} + \mathcal{O}\left(\frac{1}{kn^{2/3}}\right)
\end{align*}
We will show that all terms but the first won't contribute. We start by recalling that by

To show this we begin by recalling that $f_\omega$ has the expansion given in Lemma \ref{lemma:fexpansion} $|b_k|\le C \frac{k^2}{n^{2/3}}$. This means that 
\begin{equation}
\label{eq:bbound}
\frac{b_k^2-2b_k}{k^2-1} \le \frac{2k^2}{k^2-1} \frac{1}{n^{4/3}}.
\end{equation}

Now consider the terms inside the bracket.
\[
f_\omega'(x) = \frac{2\omega}{3} x + \frac{3}{2}x^2 + \mathcal{O}(x^3)
\]
which gives us the bound 
\[
f_{\omega_n}'(x_k)\le C  \frac{k}{n^{1/3}} 
\]
for $-\delta n^{1/3} \le k \le 2$. This together with \eqref{eq:bbound} gives us that 
\[
 \frac{2f'(x_k)}{k^2 n^{1/3}} \cdot \frac{b_k^2-2b_k}{k^2-1} \le C' \frac{k^2}{k^2-1}\frac{1}{n^{4/3}}.
\]
We can similarly show that the terms $ \frac{f''(x_k)}{k n^{2/3}}$ and $ \frac{1}{kn^{2/3}}$ will contribute errors that can be bounded uniformly by $C n^{-4/3}$ for $k = n- \lfloor \delta n^{1/3}\rfloor, \ldots, n-2$  leaving us with:
\[
\bigg[\frac{b_{k+1}+ b_{k-1}}{k}+\frac{b_{k-1}-b_{k+1}}{k^2}-\frac{2}{k}\bigg] \frac{b_k^2-2b_k}{k^2-1} = 2\frac{f(x_k)-1}{k}\cdot \frac{b_k^2-2b_k}{k^2-1} + \mathcal{O}(n^{-4/3})
\]
The result follows by applying Lemma \ref{lem:adjustingbk}.
\end{proof}

\section{Proof of Theorem \ref{thm:Fbounds}}
\label{sec:Fboundsproof}

In order to prove the bounds in Theorem \ref{thm:Fbounds} we begin with the following breakdowns:
\begin{equation}
\label{eq:split1}
 \left\|\left(F'(a^{(0)})\right)^{-1}F(a^{(0)})\right\|_\infty \le  \left\|\left(F'(a^{(0)})\right)^{-1}\right\|_\infty \left\|F(a^{(0)})\right\|_\infty
\end{equation}
\begin{equation}
\label{eq:split2}
\left\|\left(F'(a^{(0)})\right)^{-1}\left(F'(a))-F'(b)\right)\right\|_\infty \le  \left\|\left(D^{-1}F'(a^{(0)})\right)^{-1}\right\|_\infty \left\|D^{-1}\left(F'(a))-F'(b)\right)\right\|_\infty 
\end{equation}
where $D$ is a diagonal matrix with $\azero$ on the diagonal. Our computation is therefore broken down into bounding the four terms in the upper bounds. Proposition  \ref{prop:Fbounds1} from the previous section gives us the necessary bounds on $\|F(\azero)\|_\infty$. The remaining three bounds are proved below.

\begin{proposition}
\label{prop:Fbounds4}
There exists $C>0$ uniform in $k$ and $n$ such that 
\[
\| D^{-1}(F'(a)-F'(b))_{jk} \|_\infty \le C \|a-b\|_\infty.
\]
\end{proposition}
It can in fact be checked that $C\approx 4$.

\begin{proof}
We begin by computing $F'(a)-F'(b)$
\begin{align*}
(F'(a)-F'(b))_{jk} = \begin{cases}
(a_j - b_j)(a_j+b_j) & k = j\pm 1\\
2a_j(a_{j+1}+a_{j-1}) - 2b_j(b_{j+1}+b_{j-1}) & k = j\\
0 & \text{ otherwise}
\end{cases}
\end{align*}
Using that $a,b \in B(\azero, cn^{-2/3})$ we get that
\[
\frac{a_j}{\azero_j} = 1+ \mathcal{O}(n^{-2/3}) = \frac{b_j}{\azero_j}.
\]
This means that 
\[
D^{-1}(F'(a)-F'(b))_{jk} = \begin{cases}
2(a_j - b_j)(1 + \mathcal{O}(n^{-2/3}))  & k = j\pm 1\\
2\Big(( a_{j+1}- b_{j+1})+ (a_{j-1}-b_{j-1})\Big)\left(1  + \mathcal{O}(n^{-2/3})\right) & k = j\\
0 & \text{ otherwise}
\end{cases},
\]
which gives
\[
\| D^{-1}(F'(a)-F'(b))_{jk} \|_\infty \le C \|a-b\|_\infty 
\]

\end{proof}

We now consider the terms in \eqref{eq:split1} and \eqref{eq:split2} that involve a matrix inverse:
\begin{equation}
\label{eq:matrixinversenorms}
\left\|\left(F'(\azero)\right)^{-1}\right\|_\infty \quad \text{and} \quad \left\|\left(D^{-1}F'(a^{(0)})\right)^{-1}\right\|_\infty.
\end{equation}
We find that we need  the following tool:

\begin{theorem}[Varah, \cite{Varah}]
\label{thm:inversebound}
Assume $M$ is diagonally dominant by rows (i.e. $|m_{kk}| > \sum_{j\ne k} |m_{kj}|$ for all $k$) and set 
\[\gamma = \min_k \left(|m_{kk}| - \sum_{j\ne k} |m_{kj}|\right).\]
Then $\|M^{-1}\|_\infty < 1/\gamma.$
\end{theorem}
From this we can see that if we define 
\[
\gamma_k =  |m_{kk}| - \sum_{j\ne k} |m_{kj}|,
\]
where $M$ is the appropriate matrix, then giving lower bounds for $\gamma_k$ will be enough to get upper bounds on quantities in \eqref{eq:matrixinversenorms}.

\begin{proposition}
\label{prop:Fbounds2}
There exists $C>0,$ uniform in $k$ and $n$ such that $\gamma_k \ge \frac{1}{Cn^{2/3}}$, and by extension
\[
\left\|\left(F'(\azero)\right)^{-1}\right\|_\infty \le Cn^{2/3}.
\]
\end{proposition}

\begin{proof}
We compute:
\[
\left(F'(a)\right)_{kj} = \begin{cases}
1- a_k^2 &  j = k\pm1\\
-2a_k(a_{k+1} + a_{k-1}) - 2 \frac{k+\shift}{t} & j = k\\
0 & \text{ otherwise}
\end{cases}
\]
Taking $M= F'(\azero)$ we define $\gamma_k = |m_{kk}| - \sum_{j\ne k} |m_{kj}|$. Then we use that $0<\azero_k <1$ to get that
\begin{equation}
\label{eq:alphak}
\gamma_k = 2\azero_k(\azero_{k+1} + \azero_{k-1}) + 2 \frac{k+\shift}{t} - 2(1-(\azero_k)^2).
\end{equation}
Using the shorthand $x_k = \frac{k-n}{n^{1/3}}$ this can be rewritten as 
\[
\gamma_k = 2\azero_k(\azero_{k+1} + \azero_k + \azero_{k-1}) + 2 \frac{x_k+\omega_n}{n^{2/3}} + \mathcal{O}\left(\frac{1}{n}\right).
\]
Recall that we had 
\[
\azero_k = \frac{1}{n^{1/3}} \Y_{\omega_n} \left(x_k\right) - \frac{\omega_n}{3(k-n) n^{2/3}} - \frac{1}{2n},
\]
and notice that in the case where $\omega_n \ge 0$ that second term will make a positive contribution (since $k-n<0$). Because of this it is convenient to break into cases where case 1 is $\omega \ge 0$, and case 2 is $\omega<0$. Notice that for $\shift\ge 0$ we always have $\omega_n \ge \omega$.

\bigskip
\noindent \textit{Case 1: $\omega \ge 0$}

In this case we have $\azero_k  \ge \frac{1}{n^{1/3}}\Y_\omega(x_k) - \frac{1}{2n}$ which gives us that 
\[
\gamma_k  \ge \frac{2}{n^{2/3}}\Y(x_k)\big(\Y(x_{k+1}) + \Y(x_k) + \Y(x_{k-1})\big) + 2 \frac{x_k+\omega_n}{n^{2/3}} + \mathcal{O}\left(\frac{1}{n}\right).
\]
Applying a Taylor expansion to $\Y(x_{k\pm 1})$ gives 
\[
\gamma_k  \ge \frac{6}{n^{2/3}}\Y^2(x_k)+ 2 \frac{x_k+\omega_n}{n^{2/3}} + \frac{2\Y(x_k)}{n^{4/3}}\Y''(\xi) + \mathcal{O}\left(\frac{1}{n}\right).
\]
where $|\xi -x_k|<\frac{1}{n^{1/3}}$. We know from the boundary conditions of $\Y$  that 
\[
\lim_{x \to -\infty} \Y(x)\Y''(x) = 0 \quad \text{ and } \quad \lim_{x \to 0} \Y(x)\Y''(x) =+ \infty.\]
By continuity, this means that the error from the Taylor approximation can be bounded as 
$$\Y(x)\Y''(\xi) \geq -cn^{-4/3}$$ for some value of $c>0$ that is uniform in $k$ and $n$. This gives us that 
\[
\gamma_k  \ge \frac{6}{n^{2/3}}\Y^2(x_k)+ 2 \frac{x_k+\omega_n}{n^{2/3}} + \mathcal{O}\left(\frac{1}{n}\right).
\]
Now notice that if $x>-\omega_n$  both main terms are strictly positive and $\Y$ is bounded away from $0$, so this can be extended to $x\ge \omega_n$.
For $x< \omega_n$ we use Theorem \ref{thm:ubounds} which says that
\[
\Y_\omega (x)= \y_{b(\omega)}(x+\omega)> \sqrt{-(x+\omega)/3}, 
\]
for $x \in (-\infty, -\omega]$. 
Moreover we know from the other lower bound  Theorem \ref{thm:ubounds}, that the statement can in fact be made much stronger for $x$ large and negative. Therefore it suffices to use this inequality on compact intervals, meaning we have that there exists $\delta>0$ such that $\Y_\omega^2 (x) > -(1+\delta)(x+\omega)/3$, which gives us
\[
\gamma_k  > \delta\frac{6}{n^{2/3}} +  \mathcal{O}\left(\frac{1}{n}\right).
\]
Taking $n$ sufficiently large gives us that $\gamma_k$ is bounded below by $cn^{-2/3}$ for some $c>0$ as desired.

\bigskip
\noindent \textit{Case 2: $\omega < 0$}

For this case we can't simply drop the term $\frac{\omega_n}{3(k-n) n^{2/3}}$ from $\azero_k$ because it will make a negative contribution on the order of $n^{-2/3}$ when $k$ is very close to $n$. Instead we split into two regions:

Let $\delta_\omega<0$ such that for all $0> x> \delta_\omega$ we have that
\[
\Y_{\omega_n}(x) \ge \Y_\omega(x) > (1-\omega).
\]
We must have that such an $\delta_\omega$ exists since $\Y_\omega$ has pole at $x=0$. For $k-n> \delta_\omega n^{1/3}$ we then have that 
\[
\azero_k  \ge \frac{\Y(x_k)}{n^{1/3}} + \frac{\omega}{3n^{2/3}}- \frac{1}{2n} > \frac{1}{n^{1/3}} - \omega\left(\frac{1}{n^{1/3}} - \frac{1}{n^{2/3}}\right) - \frac{1}{2n} >  \frac{1}{n^{1/3}} - \frac{1}{2n}.
\]
This is enough to give us that 
\[
\gamma_k > 3\left(\frac{1}{n^{1/3}} - \frac{1}{2n}\right)^3 + \frac{x_k + \omega_k}{n^{2/3}}.
\]
Taking $n$ large enough clearly gives us that this can be bounded below by $cn^{-2/3}$ for some positive $c$. For $k-n\le \delta_\omega n^{1/3}$  we will have that  
\[
\frac{\omega_n}{3(k-n) n^{2/3}} = \mathcal{O}\left(\frac{1}{n}\right),
\]
and so this term may be omitted leaving us with
\[
\gamma_k  \ge \frac{2}{n^{2/3}}\Y(x_k)\big(\Y(x_{k+1}) + \Y(x_k) + \Y(x_{k-1})\big) + 2 \frac{x_k+\omega_n}{n^{2/3}} + \mathcal{O}\left(\frac{1}{n}\right),
\]
as before. This may be treated the same way as in Case 1.
\end{proof}

We bound $\left(D^{-1}F'(\azero )\right)^{-1}$ where $D$ is the diagonal matrix with $\azero$ on the diagonal in a similar way.
\begin{proposition}
\label{prop:Fbounds3}
Let $\gamma_k^D = |m_{kk}| - \sum_{j\ne k} |m_{kj}|$ with $M = D^{-1}F'(\azero)$. Then there exists $c>0$, uniform in $k$ and $n$ such that 
\[
\gamma_k^D \ge c n^{-1/3}.
\]
Furthermore, by Varah's Theorem \ref{thm:inversebound}, there exists $C>0,$ uniform in $k$ and $n$ such that 
\[
\left\|\left(D^{-1}F'(\azero)\right)^{-1}\right\|_\infty \le Cn^{1/3}. 
\]
\end{proposition}


\begin{proof}
We begin by computing:
\[
\left(D^{-1}F'(\azero )\right)_{kj} = \begin{cases}
\frac{1}{\azero_k}- \azero_k &  j = k\pm1\\
-2(\azero_{k+1} + \azero_{k-1}) - 2 \frac{k+\shift}{\azero_k t} & j = k\\
0 & \text{ otherwise}
\end{cases}
\]
We can use that $\azero_k>0$ to conclude that $\gamma_k^D = \gamma_k/\azero_k$, which gives
\begin{equation}
\gamma_k^D = 2(\azero_{k+1} + \azero_k +  \azero_{k-1}) + 2 \frac{x_k-\omega_n}{\azero_k n^{2/3}} +\mathcal{O}(n^{-2/3}).
\end{equation}
We now break this into two cases: The first is $k$ such that $\azero_k \le M n^{-1/3}$ for some large $M>0$. On this region we have that 
\[
\gamma_k^D = \gamma_k/\azero_k \ge \frac{c n^{-2/3}}{\azero_k} \ge \frac{c n^{-2/3}}{Mn^{-1/3}} = \frac{c}{M} n^{-1/3}.
\]
The proof for showing $\gamma_k^D > cn^{-1/3}$ can now be done in the same way as the proof that $\gamma_k > cn^{-1/3}$.

\end{proof}

We now have all the pieces necessary to prove Theorem \ref{thm:Fbounds}, which allowed us to apply the Kantorovich theorem.

\begin{proof}[Proof of Theorem \ref{thm:Fbounds}]
Proposition \ref{prop:Fbounds1}, Propositions \ref{prop:Fbounds4}, \ref{prop:Fbounds2}, and \ref{prop:Fbounds3} together with equations \eqref{eq:split1} and \eqref{eq:split2} give the bounds in Theorem \ref{thm:Fbounds}. 
\end{proof}

\end{document}